\documentclass[usenames,dvipsnames]{amsart}
\usepackage{graphicx} 
\usepackage{amsmath}
\usepackage{amsfonts}
\usepackage{amsthm}
\usepackage[dvipsnames]{xcolor}
\usepackage{xcolor}
\usepackage{tikz-cd}
\usepackage{algpseudocode}
\usepackage{algorithm}
\usetikzlibrary{decorations.markings,arrows,automata,arrows,backgrounds,snakes}

\numberwithin{algorithm}{section}

\theoremstyle{definition}
\newtheorem{theorem}{Theorem}[section]

\newtheorem{proposition}[theorem]{Proposition}
\newtheorem{definition}[theorem]{Definition}
\newtheorem{lemma}[theorem]{Lemma}
\newtheorem{remark}[theorem]{Remark}
\newtheorem{corollary}[theorem]{Corollary}
\newtheorem{example}[theorem]{Example}

\definecolor{fgreen}{RGB}{34,139,34}

\newcommand{\sd}{{{\mathrm{sd}}}}
\newcommand{\zr}{{\mathbb R}}
\newcommand{\zz}{{\mathbb Z}}

\newcommand{\zm}{{\mathbb M}}

\title{Discrete Morse theory for open complexes}

\author{Kevin P.~Knudson, Nicholas A.~Scoville}
\address{Department of Mathematics, University of Florida, Gainesville, FL 32611}
\address{Department of Mathematics and Computer Science, Ursinus College, Collegeville, PA 19426}
\date{\today}

\email{kknudson@ufl.edu}
\email{nscoville@ursinus.edu}
\thanks{The second author was supported by an AMS-Simons Research Enhancement Grant for PUI Faculty.} 



\begin{document}\tikzset{->-/.style={decoration={
  markings,
  mark=at position .5 with {{->}}},postaction={decorate}}}

\begin{abstract}
We develop a discrete Morse theory for open simplicial complexes $K=X\setminus T$ where $X$ is a simplicial complex and $T$ a subcomplex of $X$. A discrete Morse function $f$ on $K$ gives rise to a discrete Morse function on the order complex $S_K$ of $K$, and the topology change determined by $f$ on $K$ can be understood by analyzing the topology change determined by the discrete Morse function on $S_K$.  This topology change is given by a structure theorem on the level subcomplexes of $S_K$. Finally, we show that the Borel-Moore homology of $K$, a homology theory for locally compact spaces, is isomorphic to the homology induced by a gradient vector field on $K$ and deduce corresponding weak Morse inequalities. The gradient vector field on $K$ provides a novel alternative to compute Borel-Moore homology.  

\end{abstract}

\maketitle

\section{Introduction}

Robin Forman's discrete Morse theory \cite{Forman1998, Forman2022} has a wide range of applications and uses including aiding in determining homotopy type \cite{Kozlov2022, Jonsson2008, Bayer2023}, motion planning and configuration spaces \cite{Upadhyay2023, Alpert2023, Mori2011}, and computation \cite{Bauer2021,henselman2017matroid, CurryGhristNanda2016}. It has been used to solve long-standing conjectures, such as the $K(\pi,1)$ conjecture for Artin groups \cite{PaoliniSalvetti2021}, and has been put to use to study hyperplane arrangements \cite{SalvettiSettpanella2007}. There are applications in commutative algebra \cite{JollenbeckWelker2009}, geometric group theory \cite{AdiprasitoBenedetti2020}, and group theory more generally \cite{Fernandez2023,NandaTamakiTanaka2018}, and algebraic topology \cite{FarleySabalka2005}. It is safe to say that discrete Morse theory has become an indispensable tool in many disciplines.

Forman developed discrete Morse theory for a regular CW complex $X$ where the reduction of $X$ consists of a sequence of removals of  free pairs or elementary collapses.  Since its inception, many authors have developed variations of discrete Morse theory in different settings \cite{Benedetti2012, JostZhang23, KnudsonWang2022,NandaTombari2026,Kukiela2013} or by removing cells with varying structure \cite{Zaremsky2022,Fernandez2023, BauerEdelsbrunner2017}.  The purpose of our work is to develop discrete Morse theory for a more general class of simplicial complexes that contain simplices that may be missing faces, so-called open simplicial complexes.    Formally, let $X$ be a simplicial complex, $T$ a (nonempty) subcomplex, and $K=X\setminus T$. Any such $K\subseteq X$ will be called an \textbf{open simplicial complex}. One might define a more general object along these lines, but this class is broad enough to capture many cases of interest. Open simplicial complexes naturally arise as part of discrete stratified Morse theory \cite{KnudsonWang2022}. Indeed, if one has a compact Whitney stratified space $Z$ with stratification $Z_0\subset Z_1\subset\cdots\subset Z_d=Z$, where the space $Z_i\setminus Z_{i-1}$ is a (possibly empty) $i$-manifold, then there is a triangulation $T$ of $Z$, with each $T_i$ triangulating $Z_i$ \cite{Johnson1983}. The various $T_i\setminus T_{i-1}$ are then open simplicial complexes. 

Here we briefly summarize the discrete Morse theory of open simplicial complexes as developed in section \ref{sec: structure theorem}.  Given an open simplicial complex $K$ with $K=X\setminus T$, one may place a discrete Morse function (and induced gradient vector field) on $K$ in the standard way. Naively attempting to analyze the topology of $K$ by studying the level subcomplexes results in bizarre phenomena.  Example \ref{example: main example} illustrates that (i) topology change can occur with the addition of regular simplices, (ii) the attachment of a $(p+2)$-simplex can destroy a homology class in dimension $p$, and (iii) the attachment of a $(p+2)$-simplex can create a homology class in dimension $p+1$, none of which are desirable. Figure \ref{fig:pathological} shows an example of a discrete gradient on an open simplicial complex with no critical cells; if ordinary discrete Morse theory is applied here then one would conclude that the space has the homotopy type of the empty set, which is clearly absurd. To make sense of this, we pass to the order complex $S_K$ of $K$, the poset of chains of simplices of $K$. This is an honest simplicial complex (even though $K$ is not), and in fact it is a subcomplex of the barycentric subdivision of $X$; its geometric realization $|S_K|$ is therefore a subspace of the realization of $K$. More is true: $|S_K|$ is a deformation retract of $|K|$.  By extending the discrete Morse function on $K$ to a discrete Morse function on all of $X$, we may then extend further to a discrete Morse function on the subdivision $\sd(X)$ and use this to define a function $F$ on $S_K$ that interpolates the values of the original $f:K\to\zr$. This allows us to define, for a real number $a$, a level subcomplex $S_K(a)$ of $S_K$ for this new function, and these lead to associated subspaces $K(a)$ of $K$ whose realizations retract to those of the $S_K(a)$.  Theorem \ref{thm:sublevel} shows that it is in this setting where the topology change is intelligible.  Formally, if the interval $[a,b]$ does not contain any critical values of $F\colon S_K\to\zr$, then $|K(b)|$ deformation retracts to $|K(a)|$. If $[a,b]$ contains a single critical value $c$ of $F$, with corresponding critical $i$-cell $\sigma^{i}\in S_K$, then $|K(b)|\simeq |K(a)|\cup \sigma^{i}$. 

We furthermore use discrete Morse theory on open simplicial complexes to relate the homology of the Morse complex associated to a discrete gradient on $K$ to the Borel-Moore homology  $H_\bullet^{BM}(K)$ of $K$, a homology theory for non-compact spaces.  While there are several equivalent definitions of Borel-Moore homology, in our setting we have $H_\bullet^{BM}(K) \cong H_\bullet(X,T),$ where $H_\bullet(X,T)$ is simplicial homology. If we denote the Morse complex of $K$ by $\zm_\bullet(K)$,  Theorem \ref{thm:khomology} gives an isomorphism $H_i(\zm_\bullet(K))\cong H_i^{BM}(K)$ between the (discrete Morse) homology of $K$ and its Borel-Moore homology. From there we obtain immediately weak discrete Morse inequalities relating the number of critical simplices in $K$ to the Betti numbers of the Borel-Moore homology of $K$. The gradient vector field furthermore enables computation of Borel-Moore homology by constructing a chain complex whose basis consists of critical cells, with boundary maps defined by counting gradient paths between critical cells in consecutive dimensions. Several examples illustrating the computational efficiency of this approach appear in Section \ref{sec:borelmoore}.

Benedetti \cite{Benedetti2012} developed a discrete Morse theory for manifolds with boundary, and we describe here how his results relate to ours. Suppose $M$ is a triangulated $d$-manifold with boundary $\partial M$, and let $f$ be a {\em boundary-critical} discrete Morse function on $M$ (i.e., all the cells in $\partial M$ are critical for $f$). Denote the number of critical $i$-cells for $f$ in $M\setminus\partial M$ by $c_i^{\textrm{int}}$. Benedetti proved the following:
\begin{enumerate}
    \item $M$ has the same homology as a cell complex with exactly $c_d^{\textrm{int}}$ $0$-cells, $c_{d-1}^{\textrm{int}}$ $1$-cells,..., $c_1^{\textrm{int}}$ $(d-1)$-cells, and $c_0^{\textrm{int}}$ $d$-cells.
    \item If, in addition, $M$ is PL, then $M$ is homotopy equivalent to a cell complex with $c_{d-k}$ $k$-cells, $k=0,\dots ,d$.
    \item These imply the relative weak Morse inequalities $$\textrm{rank}\ H_{d-k}(M)\le c_k^{\textrm{int}}(f)$$ for all $k$.
\end{enumerate}
To prove the first two statements, Benedetti builds an explicit complex inside the dual block complex of $M$ with the requisite number of cells in each dimension. Our methods are different and do not explicitly yield the first two statements, but Theorem \ref{thm:khomology} generalizes this as much as possible. We do recover the third statement via the observation that $M$ and $M\setminus\partial M$ have the same homotopy type and that Poincar\'e duality tells us that $$H_k^{\textrm{BM}}(M)\cong H^{d-k}(M).$$ Since $H_{d-k}(M)$ and $H^{d-k}(M)$ have the same rank, we conclude the result (Corollary \ref{cor:weak}). Moreover, our result holds in the more general setting of $K=X\setminus T$ for any simplicial pair $(X,T)$.

We note also that in \cite{JostZhang23}, the authors develop a discrete Morse theory for certain types of hypergraphs, which they call {\em complex-like}. This is a technical condition, but it includes the requirement that the link of a hyperedge have the homotopy type of a sphere. Our open simplicial complexes are hypergraphs, but they need not be complex-like. Indeed, our running example, shown in Figure \ref{fig:runningexample}, is not complex-like as the link of the simplex $\sigma$ has the homotopy type of three points, which is not a sphere. Moreover, our methods here differ from those in \cite{JostZhang23}, which relates discrete Morse theory to the continuous version via Lov\'asz extension.

Finally, some remarks on computational aspects are in order. We take a discrete gradient on an open simplicial complex as given. That is, we do not explicitly say how one might obtain such an object. Existing algorithms (e.g., \cite{BabsonHersh2005,JoswigPfetsch2006,KKM2005,Lewiner2003}) for constructing discrete gradients on simplicial complexes can be adapted easily to this setting. We therefore focus here on the topological consequences of the existence of discrete gradients on open complexes.

\section{Preliminaries}

\subsection{Discrete Morse theory}
We briefly summarize the relevant concepts from discrete Morse theory. The primary reference for this is Forman's original paper \cite{Forman1998}. Throughout this paper, $X$ will denote a finite simplicial complex. If $\alpha$ is a $k$-simplex, we often write $\alpha^{(k)}$ to indicate its dimension. If $\alpha$ is a codimension-$1$ face of $\beta$ we write $\alpha < \beta$. A \textbf{discrete vector field on $X$} is a collection $V_X$ of pairs of simplices $\{\alpha^{(k)} < \beta^{(k+1)}\}$ such that each simplex of $X$ is in at most one pair in $V_X$. Pairs in $V_X$ are called \textbf{regular} and simplices in $X$ that are not in any pair in $V_X$ are called \textbf{critical}. A $k$-$(k+1)$ \textbf{$V_X$-path} is a sequence of cells
$$\alpha_0^{(k)}<\beta^{(k+1)}>\alpha_1^{(k)}<\beta_1^{(k+1)}>\cdots <\beta_{r-1}^{(k+1)}>\alpha_r^{(k)},$$ where each pair $\alpha_i<\beta_i$ is in $V_X$ and $\alpha_i\ne \alpha_{i+1}$ for $i=0,\dots, r-1$. The path is \textbf{closed} if $\alpha_0=\alpha_r$, $r>0$. We call $V_X$ a \textbf{discrete gradient} if there are no non-trivial closed $V_X$-paths.

A \textbf{discrete Morse function on $X$} is a real-valued function on the simplices of $X$, which we denote by $f\colon X\to\zr$, satisfying the following two conditions for any simplex $\alpha^{(k)}$:
\begin{enumerate}
    \item $\#\{\beta^{(k+1)}>\alpha^{(k)}: f(\beta)\le f(\alpha)\} \le 1$;
    \item $\#\{\nu^{(k-1)} < \alpha^{(k)}: f(\nu)\ge f(\alpha)\}\le 1$.
\end{enumerate}
Forman proved \cite[Lemma 2.5]{Forman1998} that these conditions are exclusive; that is, if one of the sets above is nonempty, then the other must be empty. Given such a function $f$, we have an associated discrete gradient $V_f$ defined by including a pair $\{\alpha^{(k)}<\beta^{(k+1}\}$ precisely when $f(\alpha)\ge f(\beta)$. The simplices for which both sets above are empty are then left unpaired and therefore are critical. The conditions on the values of $f$ imply that there are no closed $V_f$-paths so that this discrete vector field is indeed a gradient. Note that a critical $0$-cell corresponds to a local minimum of $f$, and vice versa, and that a critical $n$-cell ($n=\dim X$) corresponds to a local maximum of $f$.

Conversely, given a discrete gradient $V_X$ on $X$, there are infinitely many discrete Morse functions having $V_X$ as their associated gradient. One can construct such a function explicitly, and even make it self-indexing (i.e., $f(\alpha)=\dim\alpha$ for every critical cell $\alpha$, \cite[Theorem 9.3]{Forman1998}). In practice, however, one uses a result from the theory of directed graphs. Denote the Hasse diagram of $X$ by $H_X$. This is a directed graph having a vertex for each simplex of $X$ and a directed edge $\beta^{(k+1)}>\alpha^{(k)}$ for each face-pair in $X$. Given $V_X$, reverse the edges in $H_X$ corresponding to each pair $\{\alpha^{(k)}<\beta^{(k+1)}\}$ in $V_X$. Since $V_X$ has no closed paths, the resulting directed graph is acyclic, and it is a standard theorem that such directed graphs support functions on their vertices that decrease along every directed path. Such a function is a discrete Morse function on $X$ with gradient $V_X$. Note that we can always choose such a function to be injective if we wish. Moreover, and this will be important later, if we choose a collection of $V_X$-paths and specify the values of a function on the associated vertices in the modified Hasse diagram in such a way that the function values decrease along the collection of paths, then we can extend this function to the entire modified Hasse diagram to a discrete Morse function on $X$.

\subsection{Open simplicial complexes}
Let $X$ be a simplicial complex. An \textbf{open simplicial complex} $K$ is the complement in $X$ of a (nonempty) subcomplex $T$; that is, $ K=X\setminus T$. The complex $K$ is noncompact. It is the union of open simplices in $X$; that is, $K$ may contain a $k$-simplex of $X$ but not all of that simplex's faces. 

Note that we may still define a discrete gradient vector field on an open simplicial complex $K$: it is simply a discrete vector field with no non-trivial closed paths. If one tries to analogously define a discrete Morse function on $K$, there is one technical problem: the defining conditions above may no longer be exclusive. The proof of that very much relies on the fact that one is in a simplicial complex. However, given a discrete gradient on $K$, the construction of an associated discrete Morse function may be carried out on the modified Hasse diagram of $K$. (Or, one could use the definition above with the imposition of a condition that if one of the sets is nonempty, then the other must be empty.) Moreover, we no longer necessarily have a correspondence between local minima of the function and vertices of $K$; that is, a local minimum of a discrete Morse function might occur on a simplex of dimension greater than $0$.

Figure \ref{fig:runningexample} shows an example of such a $K$ along with a discrete gradient on $K$. The subcomplex $T$ is shown by the dashed lines; it consists of the boundary curve along with the two squares in the interior. Observe that $K$ does not contain any of the vertices of $X$, as these all lie in $T$. The complex $K$ has the homotopy type of a wedge of two circles. The discrete gradient vector field on $K$ has three critical cells: the $2$-simplex $\sigma$ and the two $1$-simplices $e_1$ and $e_2$. 

\begin{figure}
    
$$
\begin{tikzpicture}[scale = 1.5,decoration={markings,mark=at position 0.6 with {\arrow{triangle 60}}},]

\phantom{\node[inner sep=0pt, circle, fill=black] (1) at (1,0) [draw] {};}
\phantom{\node[inner sep=0pt, circle, fill=black] (2) at (2,0) [draw] {};}
\phantom{\node[inner sep=0pt, circle, fill=black] (3) at (5,0) [draw] {};}
\phantom{\node[inner sep=0pt, circle, fill=black] (4) at (6,0) [draw] {};}
\phantom{\node[inner sep=0pt, circle, fill=black] (5) at (0,1) [draw] {};}
\phantom{\node[inner sep=0pt, circle, fill=black] (6) at (1,1) [draw] {};}
\phantom{\node[inner sep=0pt, circle, fill=black] (7) at (2,1) [draw] {};}
\phantom{\node[inner sep=0pt, circle, fill=black] (8) at (3,1) [draw] {};}
\phantom{\node[inner sep=0pt, circle, fill=black] (9) at (4,1) [draw] {};}
\phantom{\node[inner sep=0pt, circle, fill=black] (10) at (5,1) [draw] {};}
\phantom{\node[inner sep=0pt, circle, fill=black] (11) at (6,1) [draw] {};}
\phantom{\node[inner sep=0pt, circle, fill=black] (12) at (7,1) [draw] {};}
\phantom{\node[inner sep=0pt, circle, fill=black] (13) at (0,2) [draw] {};}
\phantom{\node[inner sep=0pt, circle, fill=black] (14) at (1,2) [draw] {};}
\phantom{\node[inner sep=0pt, circle, fill=black] (15) at (2,2) [draw] {};}
\phantom{\node[inner sep=0pt, circle, fill=black] (16) at (3,2) [draw] {};}
\phantom{\node[inner sep=0pt, circle, fill=black] (17) at (4,2) [draw] {};}
\phantom{\node[inner sep=0pt, circle, fill=black] (18) at (5,2) [draw] {};}
\phantom{\node[inner sep=0pt, circle, fill=black] (19) at (6,2) [draw] {};}
\phantom{\node[inner sep=0pt, circle, fill=black] (20) at (7,2) [draw] {};}
\phantom{\node[inner sep=0pt, circle, fill=black] (21) at (1,3) [draw] {};}
\phantom{\node[inner sep=0pt, circle, fill=black] (22) at (2,3) [draw] {};}
\phantom{\node[inner sep=0pt, circle, fill=black] (23) at (5,3) [draw] {};}
\phantom{\node[inner sep=0pt, circle, fill=black] (24) at (6,3) [draw] {};}

\filldraw[fill=black!30, draw=none] (1,0)--(1,1)--(0,1)--cycle;
\filldraw[fill=black!30, draw=none] (1,0)--(1,1)--(2,1)--cycle;
\filldraw[fill=black!30, draw=none] (1,0)--(2,0)--(2,1)--cycle;
\filldraw[fill=black!30, draw=none] (2,0)--(2,1)--(3,1)--cycle;
\filldraw[fill=black!30, draw=none] (5,0)--(5,1)--(4,1)--cycle;
\filldraw[fill=black!30, draw=none] (5,0)--(5,1)--(6,1)--cycle;
\filldraw[fill=black!30, draw=none] (5,0)--(6,0)--(6,1)--cycle;
\filldraw[fill=black!30, draw=none] (6,0)--(6,1)--(7,1)--cycle;
\filldraw[fill=black!30, draw=none] (1,2)--(1,3)--(2,3)--cycle;
\filldraw[fill=black!30, draw=none] (1,2)--(2,2)--(2,3)--cycle;
\filldraw[fill=black!30, draw=none] (5,2)--(5,3)--(6,3)--cycle;
\filldraw[fill=black!30, draw=none] (5,2)--(6,2)--(6,3)--cycle;
\filldraw[fill=black!30, draw=none] (0,1)--(0,2)--(1,2)--cycle;
\filldraw[fill=black!30, draw=none] (0,1)--(1,1)--(1,2)--cycle;
\filldraw[fill=black!30, draw=none] (2,1)--(2,2)--(3,2)--cycle;
\filldraw[fill=black!30, draw=none] (2,1)--(3,1)--(3,2)--cycle;
\filldraw[fill=black!30, draw=none] (4,1)--(4,2)--(5,2)--cycle;
\filldraw[fill=black!30, draw=none] (4,1)--(5,1)--(5,2)--cycle;
\filldraw[fill=black!30, draw=none] (6,1)--(6,2)--(7,2)--cycle;
\filldraw[fill=black!30, draw=none] (6,1)--(7,1)--(7,2)--cycle;
\filldraw[fill=black!30, draw=none] (0,2)--(1,2)--(1,3)--cycle;
\filldraw[fill=black!30, draw=none] (2,2)--(2,3)--(3,2)--cycle;
\filldraw[fill=black!30, draw=none] (4,2)--(5,2)--(5,3)--cycle;
\filldraw[fill=black!30, draw=none] (6,2)--(6,3)--(7,2)--cycle;
\filldraw[fill=black!30, draw=none] (3,1)--(3,2)--(4,2)--cycle;
\filldraw[fill=black!30, draw=none] (3,1)--(4,1)--(4,2)--cycle;

\draw[dashed]  (1)--(2) node[midway, left] {};
\draw[dashed]  (2)--(8) node[midway, left] {};
\draw[dashed]  (8)--(9) node[midway, left] {};
\draw[dashed]  (9)--(3) node[midway, left] {};
\draw[dashed]  (3)--(4) node[midway, left] {};
\draw[dashed]  (4)--(12) node[midway, left] {};
\draw[dashed]  (12)--(20) node[midway, left] {};
\draw[dashed]  (20)--(24) node[midway, left] {};
\draw[dashed]  (24)--(23) node[midway, left] {};
\draw[dashed]  (23)--(17) node[midway, left] {};
\draw[dashed]  (17)--(16) node[midway, left] {};
\draw[dashed]  (16)--(22) node[midway, left] {};
\draw[dashed]  (22)--(21) node[midway, left] {};
\draw[dashed]  (21)--(13) node[midway, left] {};
\draw[dashed]  (13)--(5) node[midway, left] {};
\draw[dashed]  (5)--(1) node[midway, left] {};

\draw[dashed]  (6)--(7) node[midway, left] {};
\draw[dashed]  (7)--(15) node[midway, left] {};
\draw[dashed]  (15)--(14) node[midway, left] {};
\draw[dashed]  (14)--(6) node[midway, left] {};

\draw[dashed]  (10)--(11) node[midway, left] {};
\draw[dashed]  (11)--(19) node[midway, left] {};
\draw[dashed]  (19)--(18) node[midway, left] {};
\draw[dashed]  (18)--(10) node[midway, left] {};

\draw[-]  (1)--(6) node[midway, left] {};
\draw[-]  (1)--(7) node[midway, left] {};
\draw[-]  (2)--(7) node[midway, left] {};
\draw[-]  (5)--(6) node[midway, left] {};
\draw[-]  (5)--(14) node[midway, left] {};
\draw[-]  (13)--(14) node[midway, left] {};
\draw[-]  (7)--(8) node[midway, below] {$e_1$};
\draw[-]  (14)--(21) node[midway, left] {};
\draw[-]  (14)--(22) node[midway, left] {};
\draw[-]  (22)--(15) node[midway, left] {};
\draw[-]  (15)--(16) node[midway, left] {};
\draw[-]  (7)--(16) node[midway, left] {};
\draw[-]  (16)--(8) node[midway, left] {};
\draw[-]  (8)--(17) node[midway, left] {};
\draw[-]  (9)--(17) node[midway, left] {};
\draw[-]  (9)--(18) node[midway, left] {};
\draw[-]  (9)--(10) node[midway, left] {};
\draw[-]  (17)--(18) node[midway, above] {$e_2$};
\draw[-]  (18)--(23) node[midway, left] {};
\draw[-]  (18)--(24) node[midway, left] {};
\draw[-]  (19)--(24) node[midway, left] {};
\draw[-]  (19)--(20) node[midway, left] {};
\draw[-]  (11)--(20) node[midway, left] {};
\draw[-]  (11)--(12) node[midway, left] {};
\draw[-]  (11)--(4) node[midway, left] {};
\draw[-]  (3)--(10) node[midway, left] {};
\draw[-]  (3)--(11) node[midway, left] {};

\draw[-triangle 60]  (2.5,1.5)--(2.2,1.8);
\draw[-triangle 60]  (2.5,2)--(2.5,2.3);
\draw[-triangle 60]  (2,2.5)--(1.7,2.5);
\draw[-triangle 60]  (1.5,2.5)--(1.25,2.75);
\draw[-triangle 60]  (1,2.5)--(.7,2.5);
\draw[-triangle 60]  (.5,2)--(.5,1.7);
\draw[-triangle 60]  (.5,1.5)--(.75,1.25);
\draw[-triangle 60]  (.5,1)--(.5,.7);
\draw[-triangle 60]  (1,.5)--(1.3,.5);
\draw[-triangle 60]  (1.5,.5)--(1.75,.25);
\draw[-triangle 60]  (2,.5)--(2.3,.5);

\draw[-triangle 60]  (6.5,1)--(6.5,1.3);
\draw[-triangle 60]  (6.5,1.5)--(6.2,1.8);
\draw[-triangle 60]  (6.5,2)--(6.5,2.3);
\draw[-triangle 60]  (6,2.5)--(5.7,2.5);
\draw[-triangle 60]  (5.5,2.5)--(5.25,2.75);
\draw[-triangle 60]  (5,2.5)--(4.7,2.5);
\draw[-triangle 60]  (4.5,1.5)--(4.75,1.25);
\draw[-triangle 60]  (4.5,1)--(4.5,.7);
\draw[-triangle 60]  (5,.5)--(5.3,.5);
\draw[-triangle 60]  (5.5,.5)--(5.75,.25);
\draw[-triangle 60]  (6,.5)--(6.3,.5);

\draw[-triangle 60]  (3,1.5)--(3.3,1.5);
\draw[-triangle 60]  (3.5,1.5)--(3.8,1.2);
\draw[-triangle 60]  (4,1.5)--(4.3,1.5);

\node[anchor = west ]  at (2.6,1.4) {{$\sigma$}};
\node[anchor = west ]  at (4.2,1.7) {{$\tau$}};

\end{tikzpicture}
$$
    \caption{\label{fig:runningexample} An open simplicial complex.}
\end{figure}
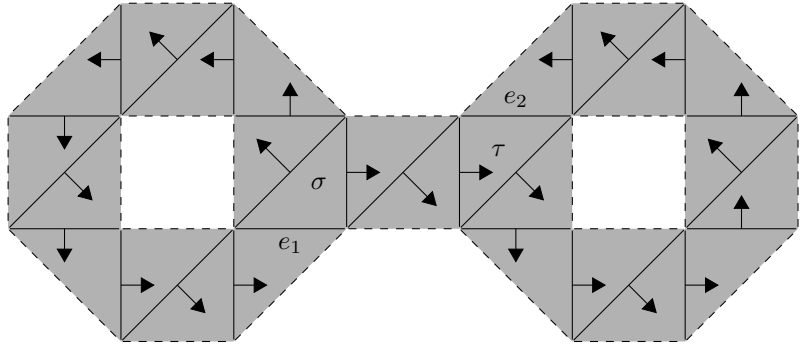

A \textbf{chain} in an open simplicial complex $K=X\setminus T$ is a collection $C$ of simplices
$$\sigma_0\subset\sigma_1\subset\cdots\subset \sigma_k$$
where each $\sigma_i\in K$. If $C$ is a maximal chain with $\sigma=\sigma_k$, we call the integer $k$ the \textbf{height} of $\sigma$. 

The \textbf{order complex} of $K$, denoted $S_K$, is the simplicial complex with vertex set equal to the set of simplices in $K$, and whose simplices are the chains in $K$: $S_K=\{C\subseteq K : C \text{ is a chain}\}$.  If $d$ denotes the maximal height of any simplex in $K$, then $S_K$ has dimension $d$. The complex $S_K$ is contained in $S_X$, the order complex of $X$.  The complex $S_X$ is simply the barycentric subdivision $\sd(X)$, and so we may realize $S_K$ as a subcomplex of $\sd(X)$ whose geometric realization $|S_K|$ lives inside of $|K|$. More generally, recall that a \textbf{hypergraph} $L=(V,E)$ is a set $V$ of \textbf{vertices} and  a set $E$ of \textbf{hyperedges}.  The geometric realization $|L|$ is the union of open simplices, one for each hyperedge of $L$, and the corresponding order complex $S_L$ is a subcomplex of the barycentric subdivision of the smallest simplicial complex $Y$ containing $L$. We state the following result from \cite{JostZhang23}, the idea of which can be traced back to Hudson's book \cite{Hudson1969} and which is also a generalization of a result of Lickorish \cite[Lemma 1]{Lickorish91}.

\begin{proposition}[\cite{JostZhang23}, Prop.~5.6]\label{prop:strongdef} Let $L$ be any hypergraph. Then the space $|S_L|$ is a strong deformation retract of $|L|$.
\end{proposition}

Our open simplicial complexes are a special kind of hypergraph. In particular, Proposition \ref{prop:strongdef} implies that if $d$ is the maximal height of a simplex in an open simplicial complex $K$, then $K$ has the homotopy type of a simplicial complex of dimension at most $d$.

Figure \ref{fig:ordercomplex} shows the order complex of $K$ from Figure \ref{fig:runningexample}. It is the one-dimensional subspace shown in the interior. The deformation retraction from $|K|$ to $|S_K|$ is evident. 

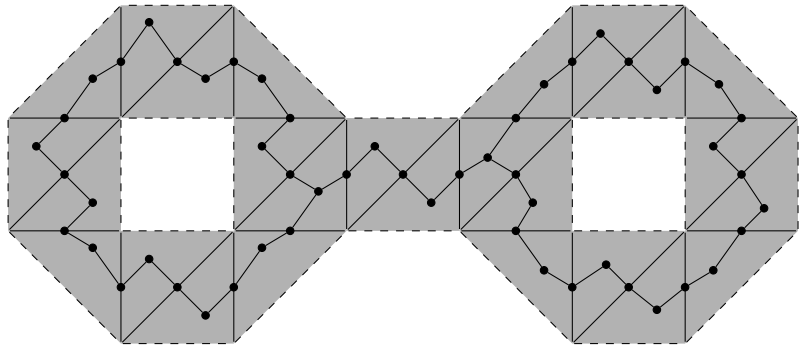
\begin{figure}
    
$$
\begin{tikzpicture}[scale = 1.5,decoration={markings,mark=at position 0.6 with {\arrow{triangle 60}}},]

\phantom{\node[inner sep=0pt, circle, fill=black] (1) at (1,0) [draw] {};}
\phantom{\node[inner sep=0pt, circle, fill=black] (2) at (2,0) [draw] {};}
\phantom{\node[inner sep=0pt, circle, fill=black] (3) at (5,0) [draw] {};}
\phantom{\node[inner sep=0pt, circle, fill=black] (4) at (6,0) [draw] {};}
\phantom{\node[inner sep=0pt, circle, fill=black] (5) at (0,1) [draw] {};}
\phantom{\node[inner sep=0pt, circle, fill=black] (6) at (1,1) [draw] {};}
\phantom{\node[inner sep=0pt, circle, fill=black] (7) at (2,1) [draw] {};}
\phantom{\node[inner sep=0pt, circle, fill=black] (8) at (3,1) [draw] {};}
\phantom{\node[inner sep=0pt, circle, fill=black] (9) at (4,1) [draw] {};}
\phantom{\node[inner sep=0pt, circle, fill=black] (10) at (5,1) [draw] {};}
\phantom{\node[inner sep=0pt, circle, fill=black] (11) at (6,1) [draw] {};}
\phantom{\node[inner sep=0pt, circle, fill=black] (12) at (7,1) [draw] {};}
\phantom{\node[inner sep=0pt, circle, fill=black] (13) at (0,2) [draw] {};}
\phantom{\node[inner sep=0pt, circle, fill=black] (14) at (1,2) [draw] {};}
\phantom{\node[inner sep=0pt, circle, fill=black] (15) at (2,2) [draw] {};}
\phantom{\node[inner sep=0pt, circle, fill=black] (16) at (3,2) [draw] {};}
\phantom{\node[inner sep=0pt, circle, fill=black] (17) at (4,2) [draw] {};}
\phantom{\node[inner sep=0pt, circle, fill=black] (18) at (5,2) [draw] {};}
\phantom{\node[inner sep=0pt, circle, fill=black] (19) at (6,2) [draw] {};}
\phantom{\node[inner sep=0pt, circle, fill=black] (20) at (7,2) [draw] {};}
\phantom{\node[inner sep=0pt, circle, fill=black] (21) at (1,3) [draw] {};}
\phantom{\node[inner sep=0pt, circle, fill=black] (22) at (2,3) [draw] {};}
\phantom{\node[inner sep=0pt, circle, fill=black] (23) at (5,3) [draw] {};}
\phantom{\node[inner sep=0pt, circle, fill=black] (24) at (6,3) [draw] {};}

\filldraw[fill=black!30, draw=none] (1,0)--(1,1)--(0,1)--cycle;
\filldraw[fill=black!30, draw=none] (1,0)--(1,1)--(2,1)--cycle;
\filldraw[fill=black!30, draw=none] (1,0)--(2,0)--(2,1)--cycle;
\filldraw[fill=black!30, draw=none] (2,0)--(2,1)--(3,1)--cycle;
\filldraw[fill=black!30, draw=none] (5,0)--(5,1)--(4,1)--cycle;
\filldraw[fill=black!30, draw=none] (5,0)--(5,1)--(6,1)--cycle;
\filldraw[fill=black!30, draw=none] (5,0)--(6,0)--(6,1)--cycle;
\filldraw[fill=black!30, draw=none] (6,0)--(6,1)--(7,1)--cycle;
\filldraw[fill=black!30, draw=none] (1,2)--(1,3)--(2,3)--cycle;
\filldraw[fill=black!30, draw=none] (1,2)--(2,2)--(2,3)--cycle;
\filldraw[fill=black!30, draw=none] (5,2)--(5,3)--(6,3)--cycle;
\filldraw[fill=black!30, draw=none] (5,2)--(6,2)--(6,3)--cycle;
\filldraw[fill=black!30, draw=none] (0,1)--(0,2)--(1,2)--cycle;
\filldraw[fill=black!30, draw=none] (0,1)--(1,1)--(1,2)--cycle;
\filldraw[fill=black!30, draw=none] (2,1)--(2,2)--(3,2)--cycle;
\filldraw[fill=black!30, draw=none] (2,1)--(3,1)--(3,2)--cycle;
\filldraw[fill=black!30, draw=none] (4,1)--(4,2)--(5,2)--cycle;
\filldraw[fill=black!30, draw=none] (4,1)--(5,1)--(5,2)--cycle;
\filldraw[fill=black!30, draw=none] (6,1)--(6,2)--(7,2)--cycle;
\filldraw[fill=black!30, draw=none] (6,1)--(7,1)--(7,2)--cycle;
\filldraw[fill=black!30, draw=none] (0,2)--(1,2)--(1,3)--cycle;
\filldraw[fill=black!30, draw=none] (2,2)--(2,3)--(3,2)--cycle;
\filldraw[fill=black!30, draw=none] (4,2)--(5,2)--(5,3)--cycle;
\filldraw[fill=black!30, draw=none] (6,2)--(6,3)--(7,2)--cycle;
\filldraw[fill=black!30, draw=none] (3,1)--(3,2)--(4,2)--cycle;
\filldraw[fill=black!30, draw=none] (3,1)--(4,1)--(4,2)--cycle;

\draw[dashed]  (1)--(2) node[midway, left] {};
\draw[dashed]  (2)--(8) node[midway, left] {};
\draw[dashed]  (8)--(9) node[midway, left] {};
\draw[dashed]  (9)--(3) node[midway, left] {};
\draw[dashed]  (3)--(4) node[midway, left] {};
\draw[dashed]  (4)--(12) node[midway, left] {};
\draw[dashed]  (12)--(20) node[midway, left] {};
\draw[dashed]  (20)--(24) node[midway, left] {};
\draw[dashed]  (24)--(23) node[midway, left] {};
\draw[dashed]  (23)--(17) node[midway, left] {};
\draw[dashed]  (17)--(16) node[midway, left] {};
\draw[dashed]  (16)--(22) node[midway, left] {};
\draw[dashed]  (22)--(21) node[midway, left] {};
\draw[dashed]  (21)--(13) node[midway, left] {};
\draw[dashed]  (13)--(5) node[midway, left] {};
\draw[dashed]  (5)--(1) node[midway, left] {};

\draw[dashed]  (6)--(7) node[midway, left] {};
\draw[dashed]  (7)--(15) node[midway, left] {};
\draw[dashed]  (15)--(14) node[midway, left] {};
\draw[dashed]  (14)--(6) node[midway, left] {};

\draw[dashed]  (10)--(11) node[midway, left] {};
\draw[dashed]  (11)--(19) node[midway, left] {};
\draw[dashed]  (19)--(18) node[midway, left] {};
\draw[dashed]  (18)--(10) node[midway, left] {};

\draw[-]  (1)--(6) node[midway, left] {};
\draw[-]  (1)--(7) node[midway, left] {};
\draw[-]  (2)--(7) node[midway, left] {};
\draw[-]  (5)--(6) node[midway, left] {};
\draw[-]  (5)--(14) node[midway, left] {};
\draw[-]  (13)--(14) node[midway, left] {};
\draw[-]  (7)--(8) node[midway, below] {};
\draw[-]  (14)--(21) node[midway, left] {};
\draw[-]  (14)--(22) node[midway, left] {};
\draw[-]  (22)--(15) node[midway, left] {};
\draw[-]  (15)--(16) node[midway, left] {};
\draw[-]  (7)--(16) node[midway, left] {};
\draw[-]  (16)--(8) node[midway, left] {};
\draw[-]  (8)--(17) node[midway, left] {};
\draw[-]  (9)--(17) node[midway, left] {};
\draw[-]  (9)--(18) node[midway, left] {};
\draw[-]  (9)--(10) node[midway, left] {};
\draw[-]  (17)--(18) node[midway, above] {};
\draw[-]  (18)--(23) node[midway, left] {};
\draw[-]  (18)--(24) node[midway, left] {};
\draw[-]  (19)--(24) node[midway, left] {};
\draw[-]  (19)--(20) node[midway, left] {};
\draw[-]  (11)--(20) node[midway, left] {};
\draw[-]  (11)--(12) node[midway, left] {};
\draw[-]  (11)--(4) node[midway, left] {};
\draw[-]  (3)--(10) node[midway, left] {};
\draw[-]  (3)--(11) node[midway, left] {};

\node[inner sep=1pt, circle, fill=black] (O0) at (4.5,2) [draw] {};
\node[inner sep=1pt, circle, fill=black] (O1) at (4.75,2.3) [draw] {};
\node[inner sep=1pt, circle, fill=black] (O2) at (5,2.5) [draw] {};
\node[inner sep=1pt, circle, fill=black] (O3) at (5.25,2.75) [draw] {};
\node[inner sep=1pt, circle, fill=black] (O4) at (5.5,2.5) [draw] {};
\node[inner sep=1pt, circle, fill=black] (O5) at (5.75,2.25) [draw] {};
\node[inner sep=1pt, circle, fill=black] (O6) at (6,2.5) [draw] {};
\node[inner sep=1pt, circle, fill=black] (O7) at (6.3,2.3) [draw] {};
\node[inner sep=1pt, circle, fill=black] (O8) at (6.5,2) [draw] {};
\node[inner sep=1pt, circle, fill=black] (O9) at (6.25,1.75) [draw] {};
\node[inner sep=1pt, circle, fill=black] (O10) at (6.5,1.5) [draw] {};
\node[inner sep=1pt, circle, fill=black] (O11) at (6.7,1.2) [draw] {};
\node[inner sep=1pt, circle, fill=black] (O12) at (6.5,1) [draw] {};
\node[inner sep=1pt, circle, fill=black] (O13) at (6.25,.65) [draw] {};
\node[inner sep=1pt, circle, fill=black] (O14) at (6,.5) [draw] {};
\node[inner sep=1pt, circle, fill=black] (O15) at (5.75,.3) [draw] {};
\node[inner sep=1pt, circle, fill=black] (O16) at (5.5,.5) [draw] {};
\node[inner sep=1pt, circle, fill=black] (O17) at (5.3,.7) [draw] {};
\node[inner sep=1pt, circle, fill=black] (O18) at (5,.5) [draw] {};
\node[inner sep=1pt, circle, fill=black] (O19) at (4.75,.65) [draw] {};
\node[inner sep=1pt, circle, fill=black] (O20) at (4.5,1) [draw] {};
\node[inner sep=1pt, circle, fill=black] (O21) at (4.65,1.25) [draw] {};
\node[inner sep=1pt, circle, fill=black] (O22) at (4.5,1.5) [draw] {};
\node[inner sep=1pt, circle, fill=black] (O23) at (4.25,1.65) [draw] {};
\node[inner sep=1pt, circle, fill=black] (O24) at (4,1.5) [draw] {};
\node[inner sep=1pt, circle, fill=black] (O25) at (3.75,1.25) [draw] {};
\node[inner sep=1pt, circle, fill=black] (O26) at (3.5,1.5) [draw] {};
\node[inner sep=1pt, circle, fill=black] (O27) at (3.25,1.75) [draw] {};
\node[inner sep=1pt, circle, fill=black] (O28) at (3,1.5) [draw] {};

\node[inner sep=1pt, circle, fill=black] (O29) at (2.5,1) [draw] {};
\node[inner sep=1pt, circle, fill=black] (O30) at (2.25,.85) [draw] {};
\node[inner sep=1pt, circle, fill=black] (O31) at (2,.5) [draw] {};
\node[inner sep=1pt, circle, fill=black] (O32) at (1.75,.25) [draw] {};
\node[inner sep=1pt, circle, fill=black] (O33) at (1.5,.5) [draw] {};
\node[inner sep=1pt, circle, fill=black] (O34) at (1.25,.75) [draw] {};
\node[inner sep=1pt, circle, fill=black] (O35) at (1,.5) [draw] {};
\node[inner sep=1pt, circle, fill=black] (O36) at (.75,.85) [draw] {};
\node[inner sep=1pt, circle, fill=black] (O37) at (.5,1) [draw] {};
\node[inner sep=1pt, circle, fill=black] (O38) at (.75,1.25) [draw] {};
\node[inner sep=1pt, circle, fill=black] (O39) at (.5,1.5) [draw] {};
\node[inner sep=1pt, circle, fill=black] (O40) at (.25,1.75) [draw] {};
\node[inner sep=1pt, circle, fill=black] (O41) at (.5,2) [draw] {};
\node[inner sep=1pt, circle, fill=black] (O42) at (.75,2.35) [draw] {};
\node[inner sep=1pt, circle, fill=black] (O43) at (1,2.5) [draw] {};
\node[inner sep=1pt, circle, fill=black] (O44) at (1.25,2.85) [draw] {};
\node[inner sep=1pt, circle, fill=black] (O45) at (1.5,2.5) [draw] {};
\node[inner sep=1pt, circle, fill=black] (O46) at (1.75,2.35) [draw] {};
\node[inner sep=1pt, circle, fill=black] (O47) at (2,2.5) [draw] {};
\node[inner sep=1pt, circle, fill=black] (O48) at (2.25,2.35) [draw] {};
\node[inner sep=1pt, circle, fill=black] (O49) at (2.5,2) [draw] {};
\node[inner sep=1pt, circle, fill=black] (O50) at (2.25,1.75) [draw] {};
\node[inner sep=1pt, circle, fill=black] (O51) at (2.5,1.5) [draw] {};
\node[inner sep=1pt, circle, fill=black] (O52) at (2.75,1.35) [draw] {};

\draw[]  (O0)--(O1) node[midway, left] {};
\draw[]  (O1)--(O2) node[midway, left] {};
\draw[]  (O2)--(O3) node[midway, left] {};
\draw[]  (O3)--(O4) node[midway, left] {};
\draw[]  (O4)--(O5) node[midway, left] {};
\draw[]  (O5)--(O6) node[midway, left] {};
\draw[]  (O6)--(O7) node[midway, left] {};
\draw[]  (O7)--(O8) node[midway, left] {};
\draw[]  (O8)--(O9) node[midway, left] {};
\draw[]  (O9)--(O10) node[midway, left] {};
\draw[]  (O10)--(O11) node[midway, left] {};
\draw[]  (O11)--(O12) node[midway, left] {};
\draw[]  (O12)--(O13) node[midway, left] {};
\draw[]  (O13)--(O14) node[midway, left] {};
\draw[]  (O14)--(O15) node[midway, left] {};
\draw[]  (O15)--(O16) node[midway, left] {};
\draw[]  (O16)--(O17) node[midway, left] {};
\draw[]  (O17)--(O18) node[midway, left] {};
\draw[]  (O18)--(O19) node[midway, left] {};
\draw[]  (O19)--(O20) node[midway, left] {};
\draw[]  (O20)--(O21) node[midway, left] {};
\draw[]  (O21)--(O22) node[midway, left] {};
\draw[]  (O22)--(O23) node[midway, left] {};
\draw[]  (O23)--(O24) node[midway, left] {};
\draw[]  (O24)--(O25) node[midway, left] {};
\draw[]  (O25)--(O26) node[midway, left] {};
\draw[]  (O26)--(O27) node[midway, left] {};
\draw[]  (O27)--(O28) node[midway, left] {};
\draw[]  (O23)--(O0) node[midway, left] {};

\draw[]  (O29)--(O30) node[midway, left] {};
\draw[]  (O30)--(O31) node[midway, left] {};
\draw[]  (O31)--(O32) node[midway, left] {};
\draw[]  (O32)--(O33) node[midway, left] {};
\draw[]  (O33)--(O34) node[midway, left] {};
\draw[]  (O34)--(O35) node[midway, left] {};
\draw[]  (O35)--(O36) node[midway, left] {};
\draw[]  (O36)--(O37) node[midway, left] {};
\draw[]  (O37)--(O38) node[midway, left] {};
\draw[]  (O38)--(O39) node[midway, left] {};
\draw[]  (O39)--(O40) node[midway, left] {};
\draw[]  (O40)--(O41) node[midway, left] {};
\draw[]  (O41)--(O42) node[midway, left] {};
\draw[]  (O42)--(O43) node[midway, left] {};
\draw[]  (O43)--(O44) node[midway, left] {};
\draw[]  (O44)--(O45) node[midway, left] {};
\draw[]  (O45)--(O46) node[midway, left] {};
\draw[]  (O46)--(O47) node[midway, left] {};
\draw[]  (O47)--(O48) node[midway, left] {};
\draw[]  (O48)--(O49) node[midway, left] {};
\draw[]  (O49)--(O50) node[midway, left] {};
\draw[]  (O50)--(O51) node[midway, left] {};
\draw[]  (O51)--(O52) node[midway, left] {};
\draw[]  (O29)--(O52) node[midway, left] {};
\draw[]  (O28)--(O52) node[midway, left] {};

\end{tikzpicture}
$$
    \caption{\label{fig:ordercomplex} The order complex of the complex from Figure \ref{fig:runningexample}.}
\end{figure}

\section{Borel--Moore homology}\label{sec:borelmoore}
Borel--Moore homology was originally developed in \cite{BorelMoore1960} as a homology theory for locally compact (in particular, non-compact) spaces. The standard reference is Chapter V of \cite{Bredon1996}. Given a space $Y$, one may consider the chain complex of {\em infinite} singular chains $\sum_{i=0}^\infty a_i\sigma_i$ that are {\em locally finite} in the sense that for any compact set $C\subset Y$ there are only finitely many $i$ for which $C\cap\ \textrm{supp}(\sigma_i)$ is nonempty. The ordinary boundary operator is well-defined here since taking boundaries does not change this finiteness condition. The corresponding homology groups $H_\bullet^{BM}(Y)$ are called the Borel--Moore homology of the space $Y$. If $Y$ is compact, then these groups are isomorphic to the singular homology groups of $Y$.

In our setting, there is a more concrete description of these groups. If $K=X\setminus T$ is an open simplicial complex, then there is an isomorphism
$$H_\bullet^{BM}(K) \cong H_\bullet(X,T),$$ where $H_\bullet(X,T)$ is the ordinary relative singular (or simplicial, in our case) homology. The same is true if $X$ is a CW-complex and $T$ is a closed subcomplex. Moreover, if $M$ is a smooth, oriented $d$-manifold, not necessarily compact, then there is a canonical isomorphism
$$H_i^{BM}(M)\cong H^{d-i}(M).$$ In particular, if $K=M\setminus \partial M$, then, via Hom-duality in ordinary singular homology, we see that $$\textrm{rank}\,H_i^{BM}(K) = \textrm{rank}\,H^{d-i}(K) = \beta_{d-i}(K).$$ One advantage of Borel--Moore homology is that if $M$ is non-compact, there is still a fundamental class $[M]\in H_n^{BM}(M)$, even though the singular homology $H_n(M)$ vanishes.

Let $V_K$ be a gradient vector field on $K$, and extend this to a gradient vector field $V_X$ on $X$ by decreeing that all simplices of $T$ are unpaired; that is, the restriction $V_T$ of $V_X$ to $T$ has all cells critical. Denote the number of critical cells of dimension $i$ in $V_K$ by $c_i^\textrm{int}$. The {\em Morse complex of $X$} is a chain complex $(\zm_\bullet(X),\partial)$ defined as follows. The basis of $\zm_i(X)$ is the set of critical $i$-cells in $V_X$. Given a $V_X$-path $\gamma(\tau,\sigma)$ from a critical $i$-cell $\tau$ to a critical $(i-1)$-cell $\sigma$, one assigns an index $i(\gamma)\in \{\pm 1\}$ according to a formula involving the fixed orientations of the cells of $X$. The details are not especially important here. One then defines $\partial(\tau)$ to be 
$$\partial(\tau) = \sum_{\gamma(\tau,\sigma)} i(\gamma)\sigma.$$ Note that if we use $\zz_2$-coefficients for homology, this simply counts the number of $V_X$-paths from $\tau$ to $\sigma$ mod $2$. By Theorem 8.2 of \cite{Forman1998}, the homology groups of $\zm_\bullet(X)$ are isomorphic to the singular homology groups of $X$. This construction makes sense over the complex $K$, but it is not clear what the homology groups of the resulting complex $\zm_\bullet(K)$ are computing. Indeed, it may be the case that there are no critical vertices in $K$, so that the $0$th homology of this complex vanishes, but $K$ will certainly consist of at least one component. Thus, we should not expect that the homology groups of $\zm_\bullet(K)$ agree with the singular homology groups of $K$.

Since $V_K$ is contained in $V_X$, and the gradient paths in $V_K$ stay inside of $K$, we have a canonical inclusion
$$\zm_\bullet(K)\to \zm_\bullet(X).$$ Moreover, since the restriction of $V_X$ to $T$ leaves every cell in $T$ critical, the boundary map in $\zm_\bullet(T)$ is simply the usual simplicial boundary map in the simplicial chain complex $C_\bullet(T)$.

\begin{theorem}\label{thm:khomology} For all $i\ge 0$, there is an isomorphism
$$H_i(\zm_\bullet(K))\cong H_i^{BM}(K).$$
\end{theorem}

\begin{proof}
Consider the short exact sequence of chain complexes
$$0\to \zm_\bullet(T)\to\zm_\bullet(X)\to \zm_\bullet(X,T)\to 0,$$ where $\zm_\bullet(X,T)$ is simply the quotient complex $\zm_\bullet(X)/\zm_\bullet(T)$. Since the homology of $\zm_\bullet(X)$ is isomorphic to the singular homology groups of $X$, the long exact homology sequence implies that the homology of $\zm_\bullet(X,T)$ is simply the ordinary relative homology $H_\bullet(X,T)$. In turn, these groups are the Borel--Moore homology groups of $K$. 

Note that the basis of $\zm_i(X,T)$ consists of those critical $i$-cells in $V_X$ that do not lie in $V_T$; this is precisely the basis of $\zm_i(K)$. Indeed, the subcomplex $\zm_\bullet(K)$ maps isomorphically to $\zm_\bullet(X,T)$ under the quotient map. Putting this all together we obtain isomorphisms $$H_i(\zm_\bullet(K)) \cong H_i(\zm_\bullet(X,T)) \cong H_i(X,T)\cong H_i^{BM}(K).$$
\end{proof}

\begin{example}
    As an almost trivial application of Theorem \ref{thm:khomology}, consider the open simplicial complex $K$ consisting of a single open $n$-cell. This is realized as the interior of the standard $n$-simplex. The only gradient on $K$ is the empty gradient. The Morse complex then consists of a copy of $\zz$ in dimension $n$ and $0$'s elsewhere. It follows that $$H_i^{\textrm{BM}}(K)=\begin{cases}
        \zz & i=n \\
        0 & i\ne n.
    \end{cases}$$
\end{example}

\begin{example}
    Consider the complex $K$ of Figure \ref{fig:runningexample} with the discrete gradient $V_K$ shown. There is a single critical $2$-cell and two critical $1$-cells. The complex $\zm_\bullet(K)$ is then
    $$0\to \zz\to \zz^2\to 0 \to 0.$$ The boundary map $\zz\to \zz^2$ is the $0$ map. This is clear modulo $2$ since there are two gradient paths from the critical $2$-simplex $\sigma$ to each of the critical edges $e_1$ and $e_2$, and one can check this integrally as well. It follows that $H_2^{BM}(K)\cong \zz$ and $H_1^{BM}(K)\cong \zz^2$.
\end{example}

\begin{example}
    Note that if $K=X\setminus T$, where $X$ is a regular CW-complex and $T$ is a subcomplex, then the Borel-Moore homology of $K$ is isomorphic to the relative homology $H_\bullet(X,T)$. Discrete Morse theory works just as well on regular CW-complexes and Theorem \ref{thm:khomology} applies to the Morse complex of a discrete gradient on such a $K$. 

    Consider the space $X=[-1,1]^n$ and let $T$ be the origin in $\zr^n$. A regular cell decomposition of $X$ may be defined by considering the obvious cubical decomposition of $\partial X$ and taking the cone of each cube with $T$. Then $K=X\setminus T$ is an open cell complex decomposing the punctured ball. Given a cell $\sigma$ in $\partial X$, define a discrete vector field $V$ on $K$ by $V=\{\{\sigma,\sigma\ast T\}: \sigma\in \partial X\}$. This is obviously acyclic and is in fact a perfect matching. The complex $\zm_\bullet(K)$ is the $0$-complex and thus $H^\textrm{BM}_\bullet(K)=0$.
\end{example}

\begin{example}
    Let $X$ be the solid torus shown in Figure \ref{fig:solidtorus}. We have illustrated a regular cell decomposition of $X$. Let $T$ be the boundary torus and set $K=X\setminus T$. The discrete gradient $V$ on $K$ is shown; it consists of three pairs beginning with the rightmost $2$-face of the blue $3$-cell paired with the adjacent $3$-cell and then pairing successive faces in a clockwise manner. There is one critical $2$-cell $\sigma$ and one critical $3$-cell $\tau$. The Morse complex $\zm_\bullet(K)$ is then
    $$0\to\zz\to\zz\to 0 \to 0\to 0.$$ There are two gradient paths from $\tau$ to $\sigma$: the direct one $\sigma<\tau$ and the other path going around the torus. The indices of these paths are opposite to each other so that the boundary map is $0$. Thus we deduce 
    $$H^{\textrm{BM}}_i(K) = \begin{cases}
        \zz & i=2,3 \\
        0 & \textrm{otherwise}.
    \end{cases} 
   $$ The interested reader is invited to triangulate this example by dividing each rectangular prism into two triangular prisms, subdividing these into tetrahedra via the standard prism operation, and then extending this discrete gradient to the triangulation in the obvious way. There are more critical cells, including one in dimension $1$, but the calculation of the homology yields the same result.
\end{example}

\begin{figure}
$$
\begin{tikzpicture}
\filldraw[red!30] (1.5,2.598)--(2.2,1.7)--(2.2,4.7)--(1.5,5.598)--(1.5,2.598);

\filldraw[blue!30] (1.5,5.598)--(5.5,5.598)--(5.5,2.598)--(3.7,1.7)--(2.2,1.7)--(2.2,4.7)--cycle;

\filldraw[black!10] (3.7,1.7)--(5.5,2.598)--(5.5,5.598)--(3.7,4.7)--cycle;

\filldraw[black!10] (3.2,0.75)--(4,0)--(4,3)--(3.2,3.75)--cycle;

\filldraw[black!10] (0,0)--(1.7,0.75)--(1.7,3.75)--(0,3)--cycle;

    \node[inner sep=2pt,circle,fill=black] (0) at (0,0) {};
    \node[inner sep=2pt,circle,fill=black] (1) at (4,0) {};
    \node[inner sep=2pt,circle,fill=black] (2) at (1.5,2.598) {};
    \node[inner sep=2pt,circle,fill=black] (3) at (5.5,2.598) {};
    \node[inner sep=2pt,circle,fill=black] (4) at (1.7,0.75) {};
    \node[inner sep=2pt,circle,fill=black] (5) at (3.2,0.75) {};
    \node[inner sep=2pt,circle,fill=black] (6) at (2.2,1.7){};
    \node[inner sep=2pt,circle,fill=black] (7) at (3.7,1.7) {};
    \node[inner sep=2pt,circle,fill=black] (10) at (0,3) {};
    \node[inner sep=2pt,circle,fill=black] (11) at (4,3) {};
    \node[inner sep=2pt,circle,fill=black] (12) at (1.5,5.598) {};
    \node[inner sep=2pt,circle,fill=black] (13) at (5.5,5.598) {};
    \node[inner sep=2pt,circle,fill=black] (14) at (1.7,3.75) {};
    \node[inner sep=2pt,circle,fill=black] (15) at (3.2,3.75) {};
    \node[inner sep=2pt,circle,fill=black] (16) at (2.2,4.7){};
    \node[inner sep=2pt,circle,fill=black] (17) at (3.7,4.7) {};

    \draw[thick] (0)--(10);
    \draw[thick] (0)--(1);
    \draw[thick] (1)--(3);
    \draw[thick,dashed] (0)--(2);
    \draw[thick,dashed] (2)--(3);
    \draw[thick,dashed] (0)--(4);
    \draw[thick,dashed] (4)--(5);
    \draw[thick,dashed] (4)--(6);
    \draw[thick,dashed] (6)--(7);
    \draw[thick,dashed] (5)--(7);
    \draw[thick,dashed] (1)--(5);
    \draw[thick,dashed] (3)--(7);
    \draw[thick,dashed] (2)--(6);

    \draw[thick] (10)--(11);
    \draw[thick] (11)--(13);
    \draw[thick] (10)--(12);
    \draw[thick] (12)--(13);
    \draw[thick] (10)--(14);
    \draw[thick] (14)--(15);
    \draw[thick] (14)--(16);
    \draw[thick] (16)--(17);
    \draw[thick] (15)--(17);
    \draw[thick] (11)--(15);
    \draw[thick] (13)--(17);
    \draw[thick] (12)--(16);

    \draw[thick,dashed] (2)--(12);
    \draw[thick] (3)--(13);
    \draw[thick,dashed] (4)--(14);
    \draw[thick,dashed] (5)--(15);
    \draw[thick,dashed] (6)--(16);
    \draw[thick,dashed] (7)--(17);
    \draw[thick] (1)--(11);

    \node at (1.8,4.7) {$\sigma$};
    \node at (3.15,5.15) {$\tau$};

    \draw[ultra thick,->,red] (4.75,3.15)--(4.25,2.25);
    \draw[ultra thick,->,red] (3.7,1.4)--(2.55,1.4);
    \draw[ultra thick,->,red,dashed] (1,1.4)--(1.5,2.3);
    
\end{tikzpicture}$$
    \caption{\label{fig:solidtorus}A discrete gradient on the interior of a solid torus. The red $2$-cell $\sigma$ and the blue $3$-cell $\tau$ are critical.}
\end{figure}

\begin{remark}
    These examples are simple, but they illustrate that discrete Morse theory is useful for computing the Borel-Moore homology of finitely triangulable spaces. Indeed, traditional methods rely on sheaf theory, complexes of locally-finite chains, or passing to the one-point compactification, each of which has its advantages and drawbacks. We have demonstrated that one may compute the homology by simply finding a discrete gradient on the open complex, which may be easier than computing the homology directly from the definition. 
\end{remark}

The following corollary, in the manifold case, was proved by Benedetti \cite[Cor. 3.4]{Benedetti2012}.

\begin{corollary}[Weak Morse inequalities]\label{cor:weak}
Let $K$ be an open simplicial complex with a discrete gradient $V_K$ and let $c_i$ be the number of critical $i$-cells in $V_K$. Then $$\textrm{rank}\,H_i^{BM}(K)\le c_i.$$ In particular, if $K=M\setminus \partial M$, where $M$ is a $d$-manifold with boundary, then $$\beta_{d-i}(M)\le c_i.$$
\end{corollary} 

\begin{proof}
    The first assertion is trivial since the ranks of the homology groups of any chain complex are bounded above by the number of basis elements. In the second case, we note that $K$ and $M$ are homotopy equivalent spaces. By duality, we have $H_i^{BM}(K)\cong H^{d-i}(K)$ and the rank of the latter group is $\beta_{d-i}(K) = \beta_{d-i}(M)$. 
\end{proof}

\section{Structure theorem}\label{sec: structure theorem}

The primary application of discrete Morse theory is that a discrete gradient determines the homotopy type of a simplicial complex; that is, if $X$ is a simplicial complex with a discrete gradient $V_X$, then $X$ has the homotopy type of a cell complex with one $i$-cell for each critical $i$-simplex of $V_X$. In fact, one can obtain an even stronger result. Suppose $f$ is a discrete Morse function on $X$. For a real number $a$ define the level subcomplex $X_a$ by $$X_a = \bigcup_{\substack{\sigma\in X\\ f(\sigma)\le a}}\bigcup_{\tau\le\sigma}\tau.$$ Then if the interval $[a,b]$ contains no critical values of $f$, the complex $X_b$ simplicially collapses onto $X_a$. If there is a single critical cell $\sigma$ with $f(\sigma)\in [a,b]$, then $X_b$ is homotopy equivalent to $X_a\cup_{\partial\sigma}\sigma$. 

Note that there is no hope of obtaining the corresponding result for an open simplicial complex $K$ since the Morse complex $\zm_\bullet(K)$ computes the Borel--Moore homology of $K$ and this does not agree with the ordinary singular homology. Even worse, there could be no critical $0$-simplices in $K$ and therefore no way to even begin building a cell complex of the same homotopy type directly from a discrete gradient on $K$. 

However, we will consider the order complex of $K$ by embedding its geometric realization inside of $K$. This relationship turns out to yield the result we seek.

\begin{example}\label{example: main example}
Consider the open simplicial complex $K$ from Figure \ref{fig:runningexample}. In Figure \ref{fig:runningfunction} we have indicated a discrete Morse function corresponding to the discrete gradient. There are infinitely many such functions, but we have fixed one for illustration. Let us try to build the level ``subcomplexes" and study how the topology changes. Recall that on a simplicial complex, the topology changes only when we reach a critical value.

\begin{figure}
$$
\begin{tikzpicture}[scale = 1.7,decoration={markings,mark=at position 0.6 with {\arrow{triangle 60}}},]

\phantom{\node[inner sep=0pt, circle, fill=black] (1) at (1,0) [draw] {};}
\phantom{\node[inner sep=0pt, circle, fill=black] (2) at (2,0) [draw] {};}
\phantom{\node[inner sep=0pt, circle, fill=black] (3) at (5,0) [draw] {};}
\phantom{\node[inner sep=0pt, circle, fill=black] (4) at (6,0) [draw] {};}
\phantom{\node[inner sep=0pt, circle, fill=black] (5) at (0,1) [draw] {};}
\phantom{\node[inner sep=0pt, circle, fill=black] (6) at (1,1) [draw] {};}
\phantom{\node[inner sep=0pt, circle, fill=black] (7) at (2,1) [draw] {};}
\phantom{\node[inner sep=0pt, circle, fill=black] (8) at (3,1) [draw] {};}
\phantom{\node[inner sep=0pt, circle, fill=black] (9) at (4,1) [draw] {};}
\phantom{\node[inner sep=0pt, circle, fill=black] (10) at (5,1) [draw] {};}
\phantom{\node[inner sep=0pt, circle, fill=black] (11) at (6,1) [draw] {};}
\phantom{\node[inner sep=0pt, circle, fill=black] (12) at (7,1) [draw] {};}
\phantom{\node[inner sep=0pt, circle, fill=black] (13) at (0,2) [draw] {};}
\phantom{\node[inner sep=0pt, circle, fill=black] (14) at (1,2) [draw] {};}
\phantom{\node[inner sep=0pt, circle, fill=black] (15) at (2,2) [draw] {};}
\phantom{\node[inner sep=0pt, circle, fill=black] (16) at (3,2) [draw] {};}
\phantom{\node[inner sep=0pt, circle, fill=black] (17) at (4,2) [draw] {};}
\phantom{\node[inner sep=0pt, circle, fill=black] (18) at (5,2) [draw] {};}
\phantom{\node[inner sep=0pt, circle, fill=black] (19) at (6,2) [draw] {};}
\phantom{\node[inner sep=0pt, circle, fill=black] (20) at (7,2) [draw] {};}
\phantom{\node[inner sep=0pt, circle, fill=black] (21) at (1,3) [draw] {};}
\phantom{\node[inner sep=0pt, circle, fill=black] (22) at (2,3) [draw] {};}
\phantom{\node[inner sep=0pt, circle, fill=black] (23) at (5,3) [draw] {};}
\phantom{\node[inner sep=0pt, circle, fill=black] (24) at (6,3) [draw] {};}

\filldraw[fill=black!30, draw=none] (1,0)--(1,1)--(0,1)--cycle;
\filldraw[fill=black!30, draw=none] (1,0)--(1,1)--(2,1)--cycle;
\filldraw[fill=black!30, draw=none] (1,0)--(2,0)--(2,1)--cycle;
\filldraw[fill=black!30, draw=none] (2,0)--(2,1)--(3,1)--cycle;
\filldraw[fill=black!30, draw=none] (5,0)--(5,1)--(4,1)--cycle;
\filldraw[fill=black!30, draw=none] (5,0)--(5,1)--(6,1)--cycle;
\filldraw[fill=black!30, draw=none] (5,0)--(6,0)--(6,1)--cycle;
\filldraw[fill=black!30, draw=none] (6,0)--(6,1)--(7,1)--cycle;
\filldraw[fill=black!30, draw=none] (1,2)--(1,3)--(2,3)--cycle;
\filldraw[fill=black!30, draw=none] (1,2)--(2,2)--(2,3)--cycle;
\filldraw[fill=black!30, draw=none] (5,2)--(5,3)--(6,3)--cycle;
\filldraw[fill=black!30, draw=none] (5,2)--(6,2)--(6,3)--cycle;
\filldraw[fill=black!30, draw=none] (0,1)--(0,2)--(1,2)--cycle;
\filldraw[fill=black!30, draw=none] (0,1)--(1,1)--(1,2)--cycle;
\filldraw[fill=black!30, draw=none] (2,1)--(2,2)--(3,2)--cycle;
\filldraw[fill=black!30, draw=none] (2,1)--(3,1)--(3,2)--cycle;
\filldraw[fill=black!30, draw=none] (4,1)--(4,2)--(5,2)--cycle;
\filldraw[fill=black!30, draw=none] (4,1)--(5,1)--(5,2)--cycle;
\filldraw[fill=black!30, draw=none] (6,1)--(6,2)--(7,2)--cycle;
\filldraw[fill=black!30, draw=none] (6,1)--(7,1)--(7,2)--cycle;
\filldraw[fill=black!30, draw=none] (0,2)--(1,2)--(1,3)--cycle;
\filldraw[fill=black!30, draw=none] (2,2)--(2,3)--(3,2)--cycle;
\filldraw[fill=black!30, draw=none] (4,2)--(5,2)--(5,3)--cycle;
\filldraw[fill=black!30, draw=none] (6,2)--(6,3)--(7,2)--cycle;
\filldraw[fill=black!30, draw=none] (3,1)--(3,2)--(4,2)--cycle;
\filldraw[fill=black!30, draw=none] (3,1)--(4,1)--(4,2)--cycle;

\draw[dashed]  (1)--(2) node[midway, left] {};
\draw[dashed]  (2)--(8) node[midway, left] {};
\draw[dashed]  (8)--(9) node[midway, left] {};
\draw[dashed]  (9)--(3) node[midway, left] {};
\draw[dashed]  (3)--(4) node[midway, left] {};
\draw[dashed]  (4)--(12) node[midway, left] {};
\draw[dashed]  (12)--(20) node[midway, left] {};
\draw[dashed]  (20)--(24) node[midway, left] {};
\draw[dashed]  (24)--(23) node[midway, left] {};
\draw[dashed]  (23)--(17) node[midway, left] {};
\draw[dashed]  (17)--(16) node[midway, left] {};
\draw[dashed]  (16)--(22) node[midway, left] {};
\draw[dashed]  (22)--(21) node[midway, left] {};
\draw[dashed]  (21)--(13) node[midway, left] {};
\draw[dashed]  (13)--(5) node[midway, left] {};
\draw[dashed]  (5)--(1) node[midway, left] {};

\draw[dashed]  (6)--(7) node[midway, left] {};
\draw[dashed]  (7)--(15) node[midway, left] {};
\draw[dashed]  (15)--(14) node[midway, left] {};
\draw[dashed]  (14)--(6) node[midway, left] {};

\draw[dashed]  (10)--(11) node[midway, left] {};
\draw[dashed]  (11)--(19) node[midway, left] {};
\draw[dashed]  (19)--(18) node[midway, left] {};
\draw[dashed]  (18)--(10) node[midway, left] {};

\draw[-]  (1)--(6) node[midway, left] {};
\draw[-]  (1)--(7) node[midway, left] {};
\draw[-]  (2)--(7) node[midway, left] {};
\draw[-]  (5)--(6) node[midway, left] {};
\draw[-]  (5)--(14) node[midway, left] {};
\draw[-]  (13)--(14) node[midway, left] {};
\draw[-]  (7)--(8) node[midway, below] {};
\draw[-]  (14)--(21) node[midway, left] {};
\draw[-]  (14)--(22) node[midway, left] {};
\draw[-]  (22)--(15) node[midway, left] {};
\draw[-]  (15)--(16) node[midway, left] {};
\draw[-]  (7)--(16) node[midway, left] {};
\draw[-]  (16)--(8) node[midway, left] {};
\draw[-]  (8)--(17) node[midway, left] {};
\draw[-]  (9)--(17) node[midway, left] {};
\draw[-]  (9)--(18) node[midway, left] {};
\draw[-]  (9)--(10) node[midway, left] {};
\draw[-]  (17)--(18) node[midway, above] {};
\draw[-]  (18)--(23) node[midway, left] {};
\draw[-]  (18)--(24) node[midway, left] {};
\draw[-]  (19)--(24) node[midway, left] {};
\draw[-]  (19)--(20) node[midway, left] {};
\draw[-]  (11)--(20) node[midway, left] {};
\draw[-]  (11)--(12) node[midway, left] {};
\draw[-]  (11)--(4) node[midway, left] {};
\draw[-]  (3)--(10) node[midway, left] {};
\draw[-]  (3)--(11) node[midway, left] {};


\node[]  at (4.5,2) {\small{$0$}};
\node[]  at  (4.75,2.3) {\small{$1$}};
\node[]  at  (5,2.5) {\small{$2$}};
\node[]  at  (5.25,2.75) {\small{$3$}};
\node[]  at  (5.5,2.5) {\small{$4$}};
\node[]  at   (5.75,2.25) {\small{$5$}};
\node[]  at   (6,2.5) {\small{$6$}};
\node[]  at   (6.3,2.3) {\small{$7$}};
\node[]  at   (6.5,2) {\small{$8$}};
\node[]  at   (6.25,1.75) {\small{$9$}};
\node[]  at   (6.5,1.5) {\small{$10$}};
\node[]  at   (6.7,1.2) {\small{$11$}};
\node[]  at   (6.5,1) {\small{$12$}};
\node[]  at   (6.25,.65) {\small{$13$}};
\node[]  at   (6,.5) {\small{$14$}};
\node[]  at  (5.75,.3) {\small{$15$}};
\node[]  at   (5.5,.5) {\small{$16$}};
\node[]  at   (5.3,.7) {\small{$17$}};
\node[]  at   (5,.5) {\small{$18$}};
\node[]  at   (4.75,.65) {\small{$19$}};
\node[]  at  (4.5,1) {\small{$20$}};
\node[]  at  (4.65,1.25) {\small{$21$}};
\node[]  at   (4.5,1.5) {\small{$22$}};
\node[]  at  (4.25,1.65) {\small{$23$}};
\node[]  at   (4,1.5) {\small{$24$}};
\node[]  at   (3.75,1.25) {\small{$25$}};
\node[]  at   (3.5,1.5) {\small{$26$}};
\node[]  at  (3.25,1.75) {\small{$27$}};
\node[]  at   (3,1.5) {\small{$28$}};

\node[]  at   (2.5,1) {\small{$29$}};
\node[]  at  (2.25,.75) {\small{$30$}};
\node[]  at  (2,.5) {\small{$31$}};
\node[]  at   (1.75,.25) {\small{$32$}};
\node[]  at   (1.5,.5) {\small{$33$}};
\node[]  at   (1.25,.75) {\small{$34$}};
\node[]  at   (1,.5) {\small{$35$}};
\node[]  at   (.75,.75) {\small{$36$}};
\node[]  at   (.5,1) {\small{$37$}};
\node[]  at   (.75,1.25) {\small{$38$}};
\node[]  at   (.5,1.5) {\small{$39$}};
\node[]  at   (.25,1.75) {\small{$40$}};
\node[]  at   (.5,2) {\small{$41$}};
\node[]  at  (.75,2.35) {\small{$42$}};
\node[]  at   (1,2.5) {\small{$43$}};
\node[]  at   (1.25,2.85) {\small{$44$}};
\node[]  at  (1.5,2.5) {\small{$45$}};
\node[]  at   (1.75,2.35) {\small{$46$}};
\node[]  at   (2,2.5) {\small{$47$}};
\node[]  at   (2.25,2.35) {\small{$48$}};
\node[]  at   (2.5,2) {\small{$49$}};
\node[]  at   (2.25,1.75) {\small{$50$}};
\node[]  at   (2.5,1.5) {\small{$51$}};
\node[]  at   (2.75,1.35) {\small{$52$}};

\end{tikzpicture}
$$
    \caption{\label{fig:runningfunction} A discrete Morse function on the complex from Figure \ref{fig:runningexample}.}
\end{figure}
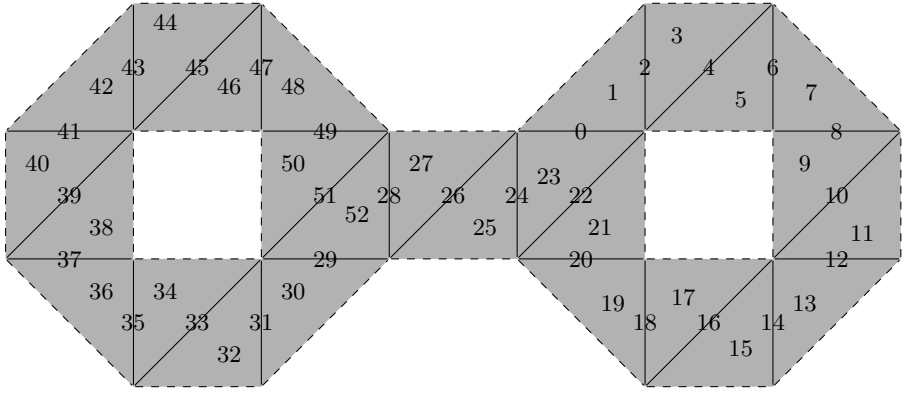

The absolute minimum occurs at the critical edge $e_2=f^{-1}(0)$. We have begun a component and we then attach cells in increasing order of function values. For a while, the homotopy type does not change. However, when we reach the cell we labeled with $\tau$ in Figure \ref{fig:runningexample}, with function value $23$, we create a $1$-cycle. This simplex is {\em regular}, and yet the topology has changed. Continuing on, nothing else changes through $f^{-1}(28)$, the edge adjacent to the critical $2$-cell $\sigma$. We then begin a new component with the critical edge $e_1=f^{-1}(29)$. We again attach cells in increasing order of function values until we finally attach $\sigma$. At that point, two things happen: the attachment of $\sigma$ destroys a $0$-dimensional homology class since the two components have merged, and it also creates a $1$-cycle.

Thus, we have seen many anomalous things. In an open simplicial complex a critical cell can create a homology class of a different dimension, or it can simultaneously create and destroy classes in multiple dimensions. Furthermore, the addition of a regular simplex can change the homotopy type. Note, however, that both $e_1$ and $e_2$ have height $0$, while $\sigma$ has height $1$.
\end{example}


The key to understanding this phenomenon is in the order complex. Recall that the order complex of $K$, $S_K$ is a subcomplex of the barycentric subdivision $\sd(X)$. We now describe how to use the discrete gradient $V_K$ on $K$ to build a discrete gradient on $S_K$ that captures these homotopy changes.

We have $K=X\setminus T$ and assume we are given a discrete gradient $V_K$ on $K$. Extend this to a discrete gradient $V_X$ on all of $X$ by making all the cells in $T$ critical, as in Section \ref{sec:borelmoore}. Choose an injective discrete Morse function $f$ on $X$ compatible with $V_X$. This choice can be arbitrary, but we fix it once and for all.

Consider the barycentric subdivision $\sd(X)$. In \cite{Zhukova2017}, A.~Zhukova describes a procedure to induce a discrete gradient on $\sd(X)$ from a given gradient on $X$ in such a way that there is a one-to-one correspondence between the critical cells of the two gradients. In fact, each critical cell in the subdivision will lie in the interior of a critical cell (of the same dimension) of the original gradient. The procedure is done on the individual simplices of $\sd(X)$ so that no closed paths are created in the interior of a simplex of $X$, and since the original vector field had no closed paths, neither does the new one. We include a description of Zhukova's algorithm in Appendix \ref{apdx}. An alternate approach to building such an object is described in \cite[Prop. 6.24]{Knudson2015}.

We will make a slight modification to the procedure in \cite{Zhukova2017} to keep better track of the critical simplices in $K$. Suppose $\sigma$ is a critical $k$-cell in $V_K$ and suppose the height of $\sigma$ is $i$.  We can choose any of the $k$-cells in $\sd(\sigma)$ to be critical, but in \cite[Sec. 3.2]{Zhukova2017} the author chooses a particular one based on a fixed ordering of the vertices of $\sigma$. If $i$ happens to equal $k$, then the entire simplex $\sd(\sigma)$ lies in $S_K$, so we can make an arbitrary choice of critical $k$-cell. The vertex ordering can be changed arbitrarily, though,  so for $i<k$ we take an ordering in which the $k$-cell we choose to be critical in $\sd(\sigma)$ has an $i$-dimensional face $\tau$ lying in $S_K$ which is paired by the procedure with an $(i+1)$-dimensional simplex in $\sd(\sigma)$. Any such simplex $\tau$ may be written as $\tau=b\ast\tau'$, where $b$ is the barycenter of $\sigma$ and $\tau'$ is a cell in the barycentric subdivision of a face $\sigma'<\sigma$ of height $i-1$. If there is more than one such $\tau$ we choose the one for which $f(\sigma)-f(\sigma')$ is the smallest (these numbers are positive; see \cite[Lemma 3.2]{JostZhang23}). This corresponds to the ``shallowest" path away from $\sigma$. 

Now consider a regular pair $\alpha^{(k)}<\beta^{(k+1)}$.  Write $\alpha=v_0v_1\cdots v_k$ and $\beta=v_0v_1\cdots v_kv_{k+1}$. If $\beta$ has no codimension-$1$ face that is the tail of a matched pair in $K$ (aside from $\alpha$, of course), perform the construction of \cite[Sec. 3.1]{Zhukova2017} to subdivide $\beta$ and pair the simplices in $\sd(\beta)$ together. The result is acyclic. 

Otherwise, there is at least one other $k$-simplex in $K$ paired with another $(k+1)$-cell in $K$. Choose the $k$-simplex $\alpha'$ in the boundary of $\beta$ of steepest descent relative to the function $f$; that is, the difference $f(\beta)-f(\alpha')$ is the greatest among all such differences. Assume we have ordered the vertices so that $\alpha' = v_1v_2\cdots v_{k+1}$ is paired with $\beta'=v_1v_2\cdots v_{k+1}u$ for some vertex $u$. Perform the construction of \cite{Zhukova2017}, and then make the following modification. Zhukova's construction pairs the barycenter $b$ with the edge $bv_{k+1}$, and it pairs the edge $bz$, where $z$ is the barycenter of $\alpha'$ with the $2$-simplex $bzv_{k+1}$. Define a new vector field $V'_\beta$ by removing the pairs $b<bv_{k+1}$ and $bz<bzv_{k+1}$ and adding the pairs $b<bz$ and $bv_{k+1}<bzv_{k+1}$. 

\begin{lemma}
    The vector field $V'_\beta$ has no closed paths.
\end{lemma}

\begin{proof}
    The original construction already produced an acyclic vector field, so it suffices to show that the new pairs do not create cycles. The pair $b<bz$ cannot be part of a cycle since the only vertex-edge pair in $\sd(\beta)$ that could possibly precede it is the pair $a<ab$, where $a$ is the barycenter of $\alpha$, and the edge $bz$ leaves $\beta$ immediately. For the other pair, note that any gradient path in the original field into the $2$-cell $bzv_{k+1}$ from an edge entered only from $bz$. In the new field such a path can enter only from $bv_{k+1}$, and since $bz$ is now paired down with $b$, there is no way to continue the path except to leave $\sd(\beta)$. So no cycles can be created in 1-2 paths. Finally, any 2-3 path involving $bzv_{k+1}$ can only end at that simplex and since all 2-3 paths were acyclic to begin with, nothing has changed.
\end{proof}

Following this procedure we obtain a discrete gradient on $\sd(X)$, which we denote by $V_X'$. Note that the vector field $V_X'$ has the property that the gradient paths in $V_K$ correspond to $V_X'$-paths in the $1$-skeleton of $S_K$. Had we not made the modification to the construction of \cite{Zhukova2017} on the regular pairs, we would have lost this information in general. Indeed, if $\alpha<\beta$ is a regular pair in $V_K$ the modification creates a path in $S_K^{(1)}$ $a\to b\to a'$, where $a$ is the barycenter of $\alpha$, $b$ is the barycenter of $\beta$ and $a'$ is the barycenter of a codimension-$1$ face of $\beta$ along which the gradient is steepest. Of course, there could be other gradient paths exiting $\beta$ through some other face $\alpha''$, but we have kept track of the steepest paths in this way.

\begin{definition}\label{def:inducedgrad}
    Suppose $K=X\setminus T$ is an open simplicial complex with discrete gradient $V_K$. Extend $V_K$ to $V_X$, choose an associated discrete Morse function, and construct the discrete gradient $V_X'$ on $\sd(X)$. The \textbf{induced gradient on $S_K$} is the gradient $W$ obtained by restricting $V_X'$ to $S_K$. Thus, a pair $\alpha^{(k)}<\beta^{(k+1)}$ lies in $W$ if and only if both $\alpha$ and $\beta$ lie in $S_K$, and a simplex $\sigma$ is critical for $W$ if either (a) $\sigma$ is critical for $V_X$ and lies in $S_K$, or (b) $\sigma$ is part of a regular pair in $V_X'$ in which the other simplex does not lie in $S_K$.
\end{definition}

That $W$ is a gradient follows from the fact that no closed paths can be created by removing pairs. We note the following about $W$.

\begin{proposition}\label{prop: gradient on S_K} Suppose $K$ is an open simplicial complex with discrete gradient $V_K$.
\begin{enumerate}
    \item If $\sigma$ is a critical $k$-cell in $V_K$ of height $i$, then $W$ has a critical $i$-cell in $S_K$, lying in the interior of $\sigma$. 
    \item There is an injection from the set of critical cells of $V_K$ to the set of critical cells of $W$.
\end{enumerate}
\end{proposition}

\begin{proof}
    For (1), this is because we arranged for there to be an $i$-simplex in $\sd(\sigma)$ paired with an $(i+1)$-simplex and since the height of $\sigma$ is $i$, this $(i+1)$-simplex cannot lie in $S_K$. Statement (2) follows at once. 
\end{proof}

\begin{example}
    To illustrate why we made the alteration to guarantee the critical cell in (1), consider the portion of an open complex shown in Figure \ref{fig:guarantee}. It shows a critical $2$-simplex $\sigma$ of height $1$ in $K$ with only one of its edges included in $K$; the remainder of the boundary of $\sigma$ lies in $T$ ($T$ is shown in red). The corresponding portion of the order complex $S_K$ is in blue; it consists of the vertex $b$ and the edge $f$. It is clear that $\sigma$ does not affect the topology of $K$ since we could deform it to the edge $e$. The figure on the left illustrates the effect of choosing  the $2$-simplex $\tau$ in $\sd(\sigma)$ to be critical and employing the procedure in \cite{Zhukova2017} without alteration. In this case we  have the barycenter $b$ paired with the edge $f$, yielding no critical cells in $W$ in the interior of $\sigma$ (since both $b$ and $f$ lie in $S_K$). In principle, this is not an issue since $W$ will determine the homotopy type of $|S_K|\simeq |K|$, but we will have lost track of the original critical cell $\sigma$ from $V_K$. Another option (figure on the right) would be to choose the cell $\tau'$ as critical, and then we would obtain two critical cells in $W$: the barycenter $b$ and the edge $f$. This is fine, of course, and since there is a unique $W$-path from $f$ to $b$ we could cancel it. We could also have chosen the $2$-simplex sharing the edge $f$ with $\tau'$ to be critical with the same effect. As we see in Example \ref{ex:runningorder} below, we can have additional critical cells appear in the interior of a $V_K$-critical cell; these may or may not be able to be canceled.
\end{example}

\begin{figure}

   $$
{\begin{tikzpicture}[scale = 1.75,decoration={markings,mark=at position 0.6 with {\arrow{triangle 60}}},]

\filldraw[red!30] (0,0)--(1.5,0)--(1.5,0.9)--(0);

\node[inner sep=2pt, circle, fill=red] (0) at (0,0) [draw] {};
\node[inner sep=2pt, circle, fill=red] (1) at (3,0) [draw] {};
\node[inner sep=2pt, circle, fill=red] (2) at (1.5,2.6) [draw] {};
\node[inner sep=2pt, circle, fill=blue!70] (3) at (1.5,.9) [draw] {};

\draw[thick, draw = red]  (0)--(1) node[midway, left] {};
\draw[thick, draw = red]  (0)--(2) node[midway, left] {};
\draw[thick, draw = black]  (0)--(3) node[midway, left] {};
\draw[thick, draw = black]  (1)--(3) node[midway, right] {};
\draw[thick, draw = black]  (2)--(3) node[midway, right] {};
\draw[thick,draw = black] (2)--(1) node[midway, right] {\huge{$e$}};
\draw[thick,draw = black] (1.5,0)--(3) node[midway, right] {};
\draw[thick,draw = black] (1.5/2,2.6/2)--(3) node[midway, right] {};
\draw[thick,draw = black] (1.5,0)--(3) node[midway, right] {};
\draw[thick,draw = blue!70,->-] (3)--(4.5/2,2.6/2) node[midway, above] {\huge{$f$}};

\node[]  at  (1,.3) {\huge{$\tau$}};
\node[anchor = north west ]  at (3) {\huge{$b$}};

\draw[thick,->] (1.5,0.8/2)--(1.9,0.8/2);
\draw[thick,->] (2.15,0.5)--(2.3,0.75);
\draw[thick,->] (0.85,0.55)--(0.85,0.9);
\draw[thick,->] (1.1,1.14)--(1.3,1.45);
\draw[thick,->] (1.5,1.6)--(1.9,1.6);

\filldraw[red!30] (5.5,0.9)--(7,0)--(6.25,2.6/2)--(5.5,0.9);

\node[inner sep=2pt, circle, fill=red] (4) at (4,0) [draw] {};
\node[inner sep=2pt, circle, fill=red] (5) at (7,0) [draw] {};
\node[inner sep=2pt, circle, fill=red] (6) at (5.5,2.6) [draw] {};
\node[inner sep=2pt, circle, fill=blue!70] (7) at (5.5,.9) [draw] {};

\draw[thick, draw = red]  (4)--(5) node[midway, left] {};
\draw[thick, draw = red]  (4)--(6) node[midway, left] {};
\draw[thick, draw = black,->-] (7)--(4) node[midway, left] {};
\draw[thick, draw = black]  (5)--(7) node[midway, right] {};
\draw[thick, draw = black]  (6)--(7) node[midway, right] {};
\draw[thick,draw = black] (6)--(5) node[midway, right] {\huge{$e$}};
\draw[thick,draw = black] (5.5,0)--(7) node[midway, right] {};
\draw[thick,draw = black] (4.75,2.6/2)--(7) node[midway, right] {};
\draw[thick,draw = black] (5.5,0)--(7) node[midway, right] {};
\draw[thick,draw = blue!70] (6.25,2.6/2)--(7) node[midway, above] {\huge{$f$}};

\node[]  at  (6.3,.8) {\huge{$\tau'$}};
\node[anchor = north west ]  at (7) {\huge{$b$}};

\draw[thick,->] (5.5,0.8/2)--(5.1,0.8/2);
\draw[thick,->] (6.15,0.5)--(6,0.25);
\draw[thick,->] (5.85,1.07)--(5.6,1.4);
\draw[thick,->] (5.1,1.14)--(4.9,0.83);
\draw[thick,->] (5.5,1.6)--(5.1,1.6);

\end{tikzpicture}} 
$$

    \caption{\label{fig:guarantee} A justification of the modification to the procedure of \cite{Zhukova2017} on $\sd(\sigma)$. On the left is the result of choosing $\tau$ to be critical; this pairs $b$ and $f$ in $W$. On the right is the result of choosing $\tau'$ to be critical; this leaves $b$ and $f$ unpaired in $W$ since they are each paired with a simplex not in $S_K$.}
\end{figure}

Note that the critical $i$-cell in (1) of Proposition \ref{prop: gradient on S_K} was chosen to enter the sublevel set filtration of $S_K$ as late as possible since we chose it to be related to the least steep path out of $\sigma$. Also, observe that in general $W$ will have more critical cells than $V_K$. This can occur, for example, when a gradient path in $V_K$ splits. Our construction forces a particular pairing in $\sd(X)$, favoring one gradient path in $V_K$, and this could lead to a simplex in another path being paired with something not in $S_K$. 

\begin{example}\label{ex:runningorder}
    Let us consider the construction applied to the gradient $V_K$ of Figure \ref{fig:runningexample}. In Figure \ref{fig:ordergradient} we show the order complex $S_K$ along with the induced gradient $W$. The critical cells are shown in red. There are two critical vertices corresponding to the original critical edges $e_1$ and $e_2$ in $V_K$. There are three critical edges. One of these corresponds to the critical $2$-simplex $\sigma$, of height $1$; we have labeled it in the diagram. The extra critical edge adjacent to the vertex corresponding to $e_1$ accounts for the additional change in topology at $\sigma$. The critical edge adjacent to the vertex corresponding to $e_2$ lies in the $2$-simplex $\tau$ whose inclusion in $K$ created a $1$-cycle. Note that the gradient path in $V_K$ splits at $\tau$ (there are two $V_K$-paths from $\sigma$ to $e_2$). 
\end{example}

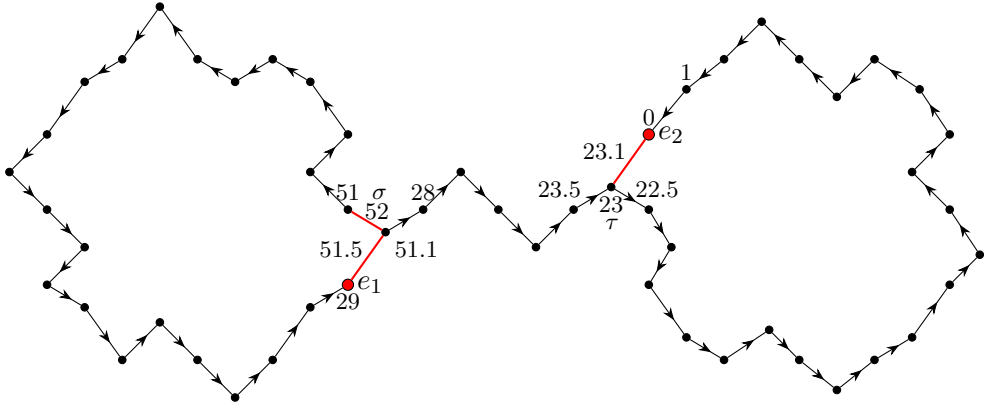
\begin{figure}
$$
\begin{tikzpicture}[scale = 2,decoration={markings,mark=at position 0.7 with {\arrow{Stealth}}},]

\node[inner sep=1.5pt, circle, fill=red] (O0) at (4.5,2) [draw] {};
\node[inner sep=1pt, circle, fill=black] (O1) at (4.75,2.3) [draw] {};
\node[inner sep=1pt, circle, fill=black] (O2) at (5,2.5) [draw] {};
\node[inner sep=1pt, circle, fill=black] (O3) at (5.25,2.75) [draw] {};
\node[inner sep=1pt, circle, fill=black] (O4) at (5.5,2.5) [draw] {};
\node[inner sep=1pt, circle, fill=black] (O5) at (5.75,2.25) [draw] {};
\node[inner sep=1pt, circle, fill=black] (O6) at (6,2.5) [draw] {};
\node[inner sep=1pt, circle, fill=black] (O7) at (6.3,2.3) [draw] {};
\node[inner sep=1pt, circle, fill=black] (O8) at (6.5,2) [draw] {};
\node[inner sep=1pt, circle, fill=black] (O9) at (6.25,1.75) [draw] {};
\node[inner sep=1pt, circle, fill=black] (O10) at (6.5,1.5) [draw] {};
\node[inner sep=1pt, circle, fill=black] (O11) at (6.7,1.2) [draw] {};
\node[inner sep=1pt, circle, fill=black] (O12) at (6.5,1) [draw] {};
\node[inner sep=1pt, circle, fill=black] (O13) at (6.25,.65) [draw] {};
\node[inner sep=1pt, circle, fill=black] (O14) at (6,.5) [draw] {};
\node[inner sep=1pt, circle, fill=black] (O15) at (5.75,.3) [draw] {};
\node[inner sep=1pt, circle, fill=black] (O16) at (5.5,.5) [draw] {};
\node[inner sep=1pt, circle, fill=black] (O17) at (5.3,.7) [draw] {};
\node[inner sep=1pt, circle, fill=black] (O18) at (5,.5) [draw] {};
\node[inner sep=1pt, circle, fill=black] (O19) at (4.75,.65) [draw] {};
\node[inner sep=1pt, circle, fill=black] (O20) at (4.5,1) [draw] {};
\node[inner sep=1pt, circle, fill=black] (O21) at (4.65,1.25) [draw] {};
\node[inner sep=1pt, circle, fill=black] (O22) at (4.5,1.5) [draw] {};
\node[inner sep=1pt, circle, fill=black] (O23) at (4.25,1.65) [draw] {};
\node[inner sep=1pt, circle, fill=black] (O24) at (4,1.5) [draw] {};
\node[inner sep=1pt, circle, fill=black] (O25) at (3.75,1.25) [draw] {};
\node[inner sep=1pt, circle, fill=black] (O26) at (3.5,1.5) [draw] {};
\node[inner sep=1pt, circle, fill=black] (O27) at (3.25,1.75) [draw] {};
\node[inner sep=1pt, circle, fill=black] (O28) at (3,1.5) [draw] {};

\node[inner sep=1.5pt, circle, fill=red] (O29) at (2.5,1) [draw] {};
\node[inner sep=1pt, circle, fill=black] (O30) at (2.25,.85) [draw] {};
\node[inner sep=1pt, circle, fill=black] (O31) at (2,.5) [draw] {};
\node[inner sep=1pt, circle, fill=black] (O32) at (1.75,.25) [draw] {};
\node[inner sep=1pt, circle, fill=black] (O33) at (1.5,.5) [draw] {};
\node[inner sep=1pt, circle, fill=black] (O34) at (1.25,.75) [draw] {};
\node[inner sep=1pt, circle, fill=black] (O35) at (1,.5) [draw] {};
\node[inner sep=1pt, circle, fill=black] (O36) at (.75,.85) [draw] {};
\node[inner sep=1pt, circle, fill=black] (O37) at (.5,1) [draw] {};
\node[inner sep=1pt, circle, fill=black] (O38) at (.75,1.25) [draw] {};
\node[inner sep=1pt, circle, fill=black] (O39) at (.5,1.5) [draw] {};
\node[inner sep=1pt, circle, fill=black] (O40) at (.25,1.75) [draw] {};
\node[inner sep=1pt, circle, fill=black] (O41) at (.5,2) [draw] {};
\node[inner sep=1pt, circle, fill=black] (O42) at (.75,2.35) [draw] {};
\node[inner sep=1pt, circle, fill=black] (O43) at (1,2.5) [draw] {};
\node[inner sep=1pt, circle, fill=black] (O44) at (1.25,2.85) [draw] {};
\node[inner sep=1pt, circle, fill=black] (O45) at (1.5,2.5) [draw] {};
\node[inner sep=1pt, circle, fill=black] (O46) at (1.75,2.35) [draw] {};
\node[inner sep=1pt, circle, fill=black] (O47) at (2,2.5) [draw] {};
\node[inner sep=1pt, circle, fill=black] (O48) at (2.25,2.35) [draw] {};
\node[inner sep=1pt, circle, fill=black] (O49) at (2.5,2) [draw] {};
\node[inner sep=1pt, circle, fill=black] (O50) at (2.25,1.75) [draw] {};
\node[inner sep=1pt, circle, fill=black] (O51) at (2.5,1.5) [draw] {};
\node[inner sep=1pt, circle, fill=black] (O52) at (2.75,1.35) [draw] {};

\draw[->-]  (O1)--(O0) node[midway, left] {};
\draw[->-]  (O2)--(O1) node[midway, left] {};
\draw[->-]  (O3)--(O2) node[midway, left] {};
\draw[->-]  (O4)--(O3) node[midway, left] {};
\draw[->-]  (O5)--(O4) node[midway, left] {};
\draw[->-]  (O6)--(O5) node[midway, left] {};
\draw[->-]  (O7)--(O6) node[midway, left] {};
\draw[->-]  (O8)--(O7) node[midway, left] {};
\draw[->-]  (O9)--(O8) node[midway, left] {};
\draw[->-]  (O10)--(O9) node[midway, left] {};
\draw[->-]  (O11)--(O10) node[midway, left] {};
\draw[->-]  (O12)--(O11) node[midway, left] {};
\draw[->-]  (O13)--(O12) node[midway, left] {};
\draw[->-]  (O14)--(O13) node[midway, left] {};
\draw[->-]  (O15)--(O14) node[midway, left] {};
\draw[->-]  (O16)--(O15) node[midway, left] {};
\draw[->-]  (O17)--(O16) node[midway, left] {};
\draw[->-]  (O18)--(O17) node[midway, left] {};
\draw[->-]  (O19)--(O18) node[midway, left] {};
\draw[->-]  (O20)--(O19) node[midway, left] {};
\draw[->-]  (O21)--(O20) node[midway, left] {};
\draw[->-]  (O22)--(O21) node[midway, left] {};
\draw[->-]  (O23)--(O22) node[midway, left] {};
\draw[->-]  (O24)--(O23) node[midway, left] {};
\draw[->-]  (O25)--(O24) node[midway, left] {};
\draw[->-]  (O26)--(O25) node[midway, left] {};
\draw[->-]  (O27)--(O26) node[midway, left] {};
\draw[->-]  (O28)--(O27) node[midway, left] {};
\draw[red, thick]  (O0)--(O23) node[midway, left] {};

\draw[->-]  (O30)--(O29) node[midway, above, left] {};
\draw[->-]  (O31)--(O30) node[midway, left] {};
\draw[->-]  (O32)--(O31) node[midway, left] {};
\draw[->-]  (O33)--(O32) node[midway, left] {};
\draw[->-]  (O34)--(O33) node[midway, left] {};
\draw[->-]  (O35)--(O34) node[midway, left] {};
\draw[->-]  (O36)--(O35) node[midway, left] {};
\draw[->-]  (O37)--(O36) node[midway, left] {};
\draw[->-]  (O38)--(O37) node[midway, left] {};
\draw[->-]  (O39)--(O38) node[midway, left] {};
\draw[->-]  (O40)--(O39) node[midway, left] {};
\draw[->-]  (O41)--(O40) node[midway, left] {};
\draw[->-]  (O42)--(O41) node[midway, left] {};
\draw[->-]  (O43)--(O42) node[midway, left] {};
\draw[->-]  (O44)--(O43) node[midway, left] {};
\draw[->-]  (O45)--(O44) node[midway, left] {};
\draw[->-]  (O46)--(O45) node[midway, left] {};
\draw[->-]  (O47)--(O46) node[midway, left] {};
\draw[->-]  (O48)--(O47) node[midway, left] {};
\draw[->-]  (O49)--(O48) node[midway, left] {};
\draw[->-]  (O50)--(O49) node[midway, left] {};
\draw[->-]  (O51)--(O50) node[midway, left] {};
\draw[draw=red, thick]  (O52)--(O51) node[midway, above] {};
\draw[red, thick]  (O52)--(O29) node[midway, left] {};
\draw[->-]  (O52)--(O28) node[midway, left] {};

\node[anchor =  west ]  at (O29) {\large{$e_1$}};
\node[anchor =  west ]  at (O0) {\large{$e_2$}};

\node[anchor =  south ]  at (O0) {\small{$0$}};
\node[anchor =  south ]  at (O1) {\small{$1$}};
\node[anchor =  north ]  at (O23) {\small{$23$}};
\node[anchor =  south ]  at (O28) {\small{$28$}};
\node[anchor =  north ]  at (O29) {\small{$29$}};
\node[anchor =  south ]  at (O51) {\small{$51$}};
\node[anchor =  north west ]  at (O52) {\small{$51.1$}};

4.25,1.65

\node[anchor =  north west ]  at (4.15,1.5) {\small{$\tau$}};
\node[anchor =  north west ]  at (2.6, 1.7) {\small{$\sigma$}};
\node[anchor =  north west ]  at (2.55, 1.6) {\small{$52$}};
\node[anchor =  north west ]  at (2.25, 1.35) {\small{$51.5$}};
\node[anchor =  north west ]  at (4, 2) {\small{$23.1$}};
\node[anchor =  north west ]  at (4.35, 1.75) {\small{$22.5$}};

\node[anchor =  north west ]  at (3.7, 1.75) {\small{$23.5$}};

\end{tikzpicture}
$$
    \caption{\label{fig:ordergradient} The induced gradient on $S_K$ with some of the same labels as Figure \ref{fig:relevant}.}
\end{figure}


Recall that we fixed a discrete Morse function $f\colon X\to\zr$ and used it to make some choices when refining the vector field $V_X$. In particular, the function $f$ determines values $f(\sigma)$ for every simplex $\sigma$ in $K$. We wish to use this to induce a discrete Morse function with gradient $W$ on $S_K$. We will accomplish this by defining a function $F:\sd(X)\to\zr$ with gradient $V_X'$ and then restricting it to $S_K$. 

If $\sigma$ is a critical $k$-cell for $V_K$ of height $i$, denote by $\sigma'$ the associated critical $i$-cell for $V_X'$ contained in the interior of $\sigma$.  Set $F(\sigma') = f(\sigma)$. For every regular simplex $\alpha$ in $V_K$, denote the corresponding vertex of $\sd(X)$ by $v_\alpha$ and set $F(v_\alpha)= f(\alpha)$. This defines a partial function on the simplices of $\sd(X)$. 

\begin{proposition}\label{prop:extend}
    The function $F$ can be extended to an injective discrete Morse function on $\sd(X)$ with gradient $V_X'$.
\end{proposition}

\begin{proof}
    The function $F$ specifies values only on some of the vertices of $\sd(X)$, and on the higher-dimensional critical simplices corresponding to critical cells of $V_K$ by making use of a discrete Morse function $f\colon X\to\zr$. One can show that if $\sigma$ is a critical simplex for an injective discrete Morse function $f$, then for any $\tau >\sigma > \nu$, $f(\tau)>f(\sigma)>f(\nu)$ (e.g. \cite[Lemma 3.2]{JostZhang23}). Thus, for any critical $i$-simplex $\alpha$ in $\sd(X)$, a descending path from $\alpha$ in the unmodified Hasse diagram ends at a vertex with a lower $f$-value, and any descending path ending at $\alpha$ came from a critical cell with a higher $f$-value. After modifying the Hasse diagram by reversing arrows corresponding to regular pairs in $V_X'$, it is therefore possible to extend $F$ to an injective function on the modified Hasse diagram compatible with these predetermined values.
\end{proof}

\begin{remark}
    Restricting the function $F:\sd(X)\to\zr$ to $S_K$, we obtain a discrete Morse function, which we still denote by $F:S_K\to \zr$, with gradient $W$. Note that $F$ may have more critical values on $S_K$ than on $\sd(X)$, but if a cell $\sigma\in S_K$ is critical for $F:\sd(X)\to\zr$, it remains critical in $S_K$. 
\end{remark}

Pseudocode for the procedure above appears in Algorithm \ref{alg: main}. We also describe Zhukova's procedure along with our modifications in Algorithms \ref{alg: critical} and \ref{alg: regularpair}. These may be found in Appendix \ref{apdx}.

Figure \ref{fig:fullhasse} shows the full Hasse diagram of $S_K$ from Figure \ref{fig:ordergradient}. The preset values of $F$ on the vertices corresponding to the regular simplices are shown in black and those corresponding to the higher dimensional critical faces are shown in red. There is only one predetermined value on a critical edge, namely the edge corresponding to the critical $2$-simplex $\sigma$, which is given the value $52$. Much of this diagram can be collapsed away by following the gradient. Proposition \ref{prop:extend} asserts that we can extend this function to a discrete Morse function on $\sd(X)$ and then restrict to a function on $S_K$.   The relevant simplices for one possible extension are shown in Figure \ref{fig:relevant}. The values in green are generated by extending the preset values; these are not unique, of course, but they do give a discrete Morse function on $S_K$.

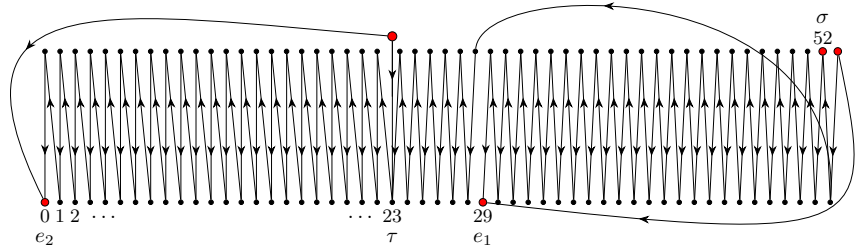
\begin{figure}
$$\resizebox{\textwidth}{!}{
\begin{tikzpicture}[scale = 2.5,decoration={markings, mark=at position 0.7 with {\arrow{Stealth}}},]

\node[inner sep=1.25pt, circle, fill=red] (00) at (0.0,0) [draw] {};
\node[inner sep=.75pt, circle, fill=black] (01) at (.1,0) [draw] {};
\node[inner sep=.75pt, circle, fill=black] (02) at (.2,0) [draw] {};
\node[inner sep=.75pt, circle, fill=black] (03) at (.3,0) [draw] {};
\node[inner sep=.75pt, circle, fill=black] (04) at (.4,0) [draw] {};
\node[inner sep=.75pt, circle, fill=black] (05) at (.5,0) [draw] {};
\node[inner sep=.75pt, circle, fill=black] (06) at (.6,0) [draw] {};
\node[inner sep=.75pt, circle, fill=black] (07) at (.7,0) [draw] {};
\node[inner sep=.75pt, circle, fill=black] (08) at (.8,0) [draw] {};
\node[inner sep=.75pt, circle, fill=black] (09) at (.9,0) [draw] {};
\node[inner sep=.75pt, circle, fill=black] (010) at (1.0,0) [draw] {};
\node[inner sep=.75pt, circle, fill=black] (011) at (1.1,0) [draw] {};
\node[inner sep=.75pt, circle, fill=black] (012) at (1.2,0) [draw] {};
\node[inner sep=.75pt, circle, fill=black] (013) at (1.3,0) [draw] {};
\node[inner sep=.75pt, circle, fill=black] (014) at (1.4,0) [draw] {};
\node[inner sep=.75pt, circle, fill=black] (015) at (1.5,0) [draw] {};
\node[inner sep=.75pt, circle, fill=black] (016) at (1.6,0) [draw] {};
\node[inner sep=.75pt, circle, fill=black] (017) at (1.7,0) [draw] {};
\node[inner sep=.75pt, circle, fill=black] (018) at (1.8,0) [draw] {};
\node[inner sep=.75pt, circle, fill=black] (019) at (1.9,0) [draw] {};
\node[inner sep=.75pt, circle, fill=black] (020) at (2.0,0) [draw] {};
\node[inner sep=.75pt, circle, fill=black] (021) at (2.1,0) [draw] {};
\node[inner sep=.75pt, circle, fill=black] (022) at (2.2,0) [draw] {};
\node[inner sep=.75pt, circle, fill=black] (023) at (2.3,0) [draw] {};
\node[inner sep=.75pt, circle, fill=black] (024) at (2.4,0) [draw] {};
\node[inner sep=.75pt, circle, fill=black] (025) at (2.5,0) [draw] {};
\node[inner sep=.75pt, circle, fill=black] (026) at (2.6,0) [draw] {};
\node[inner sep=.75pt, circle, fill=black] (027) at (2.7,0) [draw] {};
\node[inner sep=.75pt, circle, fill=black] (028) at (2.8,0) [draw] {};
\node[inner sep=1.255pt, circle, fill=red] (029) at (2.9,0) [draw] {};
\node[inner sep=.75pt, circle, fill=black] (030) at (3.0,0) [draw] {};
\node[inner sep=.75pt, circle, fill=black] (031) at (3.1,0) [draw] {};
\node[inner sep=.75pt, circle, fill=black] (032) at (3.2,0) [draw] {};
\node[inner sep=.75pt, circle, fill=black] (033) at (3.3,0) [draw] {};
\node[inner sep=.75pt, circle, fill=black] (034) at (3.4,0) [draw] {};
\node[inner sep=.75pt, circle, fill=black] (035) at (3.5,0) [draw] {};
\node[inner sep=.75pt, circle, fill=black] (036) at (3.6,0) [draw] {};
\node[inner sep=.75pt, circle, fill=black] (037) at (3.7,0) [draw] {};
\node[inner sep=.75pt, circle, fill=black] (038) at (3.8,0) [draw] {};
\node[inner sep=.75pt, circle, fill=black] (039) at (3.9,0) [draw] {};
\node[inner sep=.75pt, circle, fill=black] (040) at (4.0,0) [draw] {};
\node[inner sep=.75pt, circle, fill=black] (041) at (4.1,0) [draw] {};
\node[inner sep=.75pt, circle, fill=black] (042) at (4.2,0) [draw] {};
\node[inner sep=.75pt, circle, fill=black] (043) at (4.3,0) [draw] {};
\node[inner sep=.75pt, circle, fill=black] (044) at (4.4,0) [draw] {};
\node[inner sep=.75pt, circle, fill=black] (045) at (4.5,0) [draw] {};
\node[inner sep=.75pt, circle, fill=black] (046) at (4.6,0) [draw] {};
\node[inner sep=.75pt, circle, fill=black] (047) at (4.7,0) [draw] {};
\node[inner sep=.75pt, circle, fill=black] (048) at (4.8,0) [draw] {};
\node[inner sep=.75pt, circle, fill=black] (049) at (4.9,0) [draw] {};
\node[inner sep=.75pt, circle, fill=black] (050) at (5.0,0) [draw] {};
\node[inner sep=.75pt, circle, fill=black] (051) at (5.1,0) [draw] {};
\node[inner sep=.75pt, circle, fill=black] (052) at (5.2,0) [draw] {};

\node[inner sep=.75pt, circle, fill=black] (10) at (0.0,1) [draw] {};
\node[inner sep=.75pt, circle, fill=black] (11) at (.1,1) [draw] {};
\node[inner sep=.75pt, circle, fill=black] (12) at (.2,1) [draw] {};
\node[inner sep=.75pt, circle, fill=black] (13) at (.3,1) [draw] {};
\node[inner sep=.75pt, circle, fill=black] (14) at (.4,1) [draw] {};
\node[inner sep=.75pt, circle, fill=black] (15) at (.5,1) [draw] {};
\node[inner sep=.75pt, circle, fill=black] (16) at (.6,1) [draw] {};
\node[inner sep=.75pt, circle, fill=black] (17) at (.7,1) [draw] {};
\node[inner sep=.75pt, circle, fill=black] (18) at (.8,1) [draw] {};
\node[inner sep=.75pt, circle, fill=black] (19) at (.9,1) [draw] {};
\node[inner sep=.75pt, circle, fill=black] (110) at (1.0,1) [draw] {};
\node[inner sep=.75pt, circle, fill=black] (111) at (1.1,1) [draw] {};
\node[inner sep=.75pt, circle, fill=black] (112) at (1.2,1) [draw] {};
\node[inner sep=.75pt, circle, fill=black] (113) at (1.3,1) [draw] {};
\node[inner sep=.75pt, circle, fill=black] (114) at (1.4,1) [draw] {};
\node[inner sep=.75pt, circle, fill=black] (115) at (1.5,1) [draw] {};
\node[inner sep=.75pt, circle, fill=black] (116) at (1.6,1) [draw] {};
\node[inner sep=.75pt, circle, fill=black] (117) at (1.7,1) [draw] {};
\node[inner sep=.75pt, circle, fill=black] (118) at (1.8,1) [draw] {};
\node[inner sep=.75pt, circle, fill=black] (119) at (1.9,1) [draw] {};
\node[inner sep=.75pt, circle, fill=black] (120) at (2.0,1) [draw] {};
\node[inner sep=.75pt, circle, fill=black] (121) at (2.1,1) [draw] {};
\node[inner sep=.75pt, circle, fill=black] (122) at (2.2,1) [draw] {};
\node[inner sep=1.5pt, circle, fill=red] (123R) at (2.3,1.1) [draw] {};
\node[inner sep=.75pt, circle, fill=black] (123) at (2.3+.05,1) [draw] {};
\node[inner sep=.75pt, circle, fill=black] (124) at (2.4+.05,1) [draw] {};
\node[inner sep=.75pt, circle, fill=black] (125) at (2.5+.05,1) [draw] {};
\node[inner sep=.75pt, circle, fill=black] (126) at (2.6+.05,1) [draw] {};
\node[inner sep=.75pt, circle, fill=black] (127) at (2.7+.05,1) [draw] {};
\node[inner sep=.75pt, circle, fill=black] (128) at (2.8+.05,1) [draw] {};
\node[inner sep=.75pt, circle, fill=black] (129) at (2.9+.05,1) [draw] {};
\node[inner sep=.75pt, circle, fill=black] (130) at (3.0+.05,1) [draw] {};
\node[inner sep=.75pt, circle, fill=black] (131) at (3.1+.05,1) [draw] {};
\node[inner sep=.75pt, circle, fill=black] (132) at (3.2+.05,1) [draw] {};
\node[inner sep=.75pt, circle, fill=black] (133) at (3.3+.05,1) [draw] {};
\node[inner sep=.75pt, circle, fill=black] (134) at (3.4+.05,1) [draw] {};
\node[inner sep=.75pt, circle, fill=black] (135) at (3.5+.05,1) [draw] {};
\node[inner sep=.75pt, circle, fill=black] (136) at (3.6+.05,1) [draw] {};
\node[inner sep=.75pt, circle, fill=black] (137) at (3.7+.05,1) [draw] {};
\node[inner sep=.75pt, circle, fill=black] (138) at (3.8+.05,1) [draw] {};
\node[inner sep=.75pt, circle, fill=black] (139) at (3.9+.05,1) [draw] {};
\node[inner sep=.75pt, circle, fill=black] (140) at (4.0+.05,1) [draw] {};
\node[inner sep=.75pt, circle, fill=black] (141) at (4.1+.05,1) [draw] {};
\node[inner sep=.75pt, circle, fill=black] (142) at (4.2+.05,1) [draw] {};
\node[inner sep=.75pt, circle, fill=black] (143) at (4.3+.05,1) [draw] {};
\node[inner sep=.75pt, circle, fill=black] (144) at (4.4+.05,1) [draw] {};
\node[inner sep=.75pt, circle, fill=black] (145) at (4.5+.05,1) [draw] {};
\node[inner sep=.75pt, circle, fill=black] (146) at (4.6+.05,1) [draw] {};
\node[inner sep=.75pt, circle, fill=black] (147) at (4.7+.05,1) [draw] {};
\node[inner sep=.75pt, circle, fill=black] (148) at (4.8+.05,1) [draw] {};
\node[inner sep=.75pt, circle, fill=black] (149) at (4.9+.05,1) [draw] {};
\node[inner sep=.75pt, circle, fill=black] (150) at (5.0+.05,1) [draw] {};
\node[inner sep=1.25pt, circle, fill=red] (151) at (5.1+.05,1) [draw] {};
\node[inner sep=1.25pt, circle, fill=red] (152) at (5.2+.05,1) [draw] {};

Order complex edges
\draw[->-]  (10)--(00) node[midway, left] {};
\draw[->-]  (01)--(10) node[midway, left] {};
\draw[->-]  (11)--(01) node[midway, left] {};
\draw[->-]  (02)--(11) node[midway, left] {};
\draw[->-]  (12)--(02) node[midway, left] {};
\draw[->-]  (03)--(12) node[midway, left] {};
\draw[->-]  (13)--(03) node[midway, left] {};
\draw[->-]  (04)--(13) node[midway, left] {};
\draw[->-]  (14)--(04) node[midway, left] {};
\draw[->-]  (05)--(14) node[midway, left] {};
\draw[->-]  (15)--(05) node[midway, left] {};
\draw[->-]  (06)--(15) node[midway, left] {};
\draw[->-]  (16)--(06) node[midway, left] {};
\draw[->-]  (07)--(16) node[midway, left] {};
\draw[->-]  (17)--(07) node[midway, left] {};
\draw[->-]  (08)--(17) node[midway, left] {};
\draw[->-]  (18)--(08) node[midway, left] {};
\draw[->-]  (09)--(18) node[midway, left] {};
\draw[->-]  (19)--(09) node[midway, left] {};
\draw[->-]  (010)--(19) node[midway, left] {};
\draw[->-]  (110)--(010) node[midway, left] {};
\draw[->-]  (011)--(110) node[midway, left] {};
\draw[->-]  (111)--(011) node[midway, left] {};
\draw[->-]  (012)--(111) node[midway, left] {};
\draw[->-]  (112)--(012) node[midway, left] {};
\draw[->-]  (013)--(112) node[midway, left] {};
\draw[->-]  (113)--(013) node[midway, left] {};
\draw[->-]  (014)--(113) node[midway, left] {};
\draw[->-]  (114)--(014) node[midway, left] {};
\draw[->-]  (015)--(114) node[midway, left] {};
\draw[->-]  (115)--(015) node[midway, left] {};
\draw[->-]  (016)--(115) node[midway, left] {};
\draw[->-]  (116)--(016) node[midway, left] {};
\draw[->-]  (017)--(116) node[midway, left] {};
\draw[->-]  (117)--(017) node[midway, left] {};
\draw[->-]  (018)--(117) node[midway, left] {};
\draw[->-]  (118)--(018) node[midway, left] {};
\draw[->-]  (019)--(118) node[midway, left] {};
\draw[->-]  (119)--(019) node[midway, left] {};
\draw[->-]  (020)--(119) node[midway, left] {};
\draw[->-]  (120)--(020) node[midway, left] {};
\draw[->-]  (021)--(120) node[midway, left] {};
\draw[->-]  (121)--(021) node[midway, left] {};
\draw[->-]  (022)--(121) node[midway, left] {};
\draw[->-]  (122)--(022) node[midway, left] {};
\draw[->-]  (023)--(122) node[midway, left] {};
\draw[->-]  (123)--(023) node[midway, left] {};
\draw[-]  (123R)--(023) node[midway, left] {};
\draw[->-]  (123R)--(2.3,.7) node[midway, left] {};
\draw[->-]  (024)--(123) node[midway, left] {};
\draw[->-]  (124)--(024) node[midway, left] {};
\draw[->-]  (025)--(124) node[midway, left] {};
\draw[->-]  (125)--(025) node[midway, left] {};
\draw[->-]  (026)--(125) node[midway, left] {};
\draw[->-]  (126)--(026) node[midway, left] {};
\draw[->-]  (027)--(126) node[midway, left] {};
\draw[->-]  (127)--(027) node[midway, left] {};
\draw[->-]  (028)--(127) node[midway, left] {};
\draw[->-]  (128)--(028) node[midway, left] {};

\draw[->-]  (129)--(029) node[midway, left] {};
\draw[->-]  (030)--(129) node[midway, left] {};
\draw[->-]  (130)--(030) node[midway, left] {};
\draw[->-]  (031)--(130) node[midway, left] {};
\draw[->-]  (131)--(031) node[midway, left] {};
\draw[->-]  (032)--(131) node[midway, left] {};
\draw[->-]  (132)--(032) node[midway, left] {};
\draw[->-]  (033)--(132) node[midway, left] {};
\draw[->-]  (133)--(033) node[midway, left] {};
\draw[->-]  (034)--(133) node[midway, left] {};
\draw[->-]  (134)--(034) node[midway, left] {};
\draw[->-]  (035)--(134) node[midway, left] {};
\draw[->-]  (135)--(035) node[midway, left] {};
\draw[->-]  (036)--(135) node[midway, left] {};
\draw[->-]  (136)--(036) node[midway, left] {};
\draw[->-]  (037)--(136) node[midway, left] {};
\draw[->-]  (137)--(037) node[midway, left] {};
\draw[->-]  (038)--(137) node[midway, left] {};
\draw[->-]  (138)--(038) node[midway, left] {};
\draw[->-]  (039)--(138) node[midway, left] {};
\draw[->-]  (139)--(039) node[midway, left] {};
\draw[->-]  (040)--(139) node[midway, left] {};
\draw[->-]  (140)--(040) node[midway, left] {};
\draw[->-]  (041)--(140) node[midway, left] {};
\draw[->-]  (141)--(041) node[midway, left] {};
\draw[->-]  (042)--(141) node[midway, left] {};
\draw[->-]  (142)--(042) node[midway, left] {};
\draw[->-]  (043)--(142) node[midway, left] {};
\draw[->-]  (143)--(043) node[midway, left] {};
\draw[->-]  (044)--(143) node[midway, left] {};
\draw[->-]  (144)--(044) node[midway, left] {};
\draw[->-]  (045)--(144) node[midway, left] {};
\draw[->-]  (145)--(045) node[midway, left] {};
\draw[->-]  (046)--(145) node[midway, left] {};
\draw[->-]  (146)--(046) node[midway, left] {};
\draw[->-]  (047)--(146) node[midway, left] {};
\draw[->-]  (147)--(047) node[midway, left] {};
\draw[->-]  (048)--(147) node[midway, left] {};
\draw[->-]  (148)--(048) node[midway, left] {};
\draw[->-]  (049)--(148) node[midway, left] {};
\draw[->-]  (149)--(049) node[midway, left] {};
\draw[->-]  (050)--(149) node[midway, left] {};
\draw[->-]  (150)--(050) node[midway, left] {};
\draw[->-]  (051)--(150) node[midway, left] {};
\draw[->-]  (151)--(051) node[midway, left] {};
\draw[->-]  (052)--(151) node[midway, left] {};
\draw[->-]  (152)--(052) node[midway, left] {};

\draw[->-] (123R) .. controls (.5,1.25) and (-.75, 1.5)  .. (00);
\draw[->-] (052) .. controls (5.3,1.5) and (2.9,1.5)  .. (128);
\draw[->-] (152) .. controls (5.5,-.2) and (5.5,-.3)  .. (029);

\node[anchor =  north ]  at (00) {\small{0}};
\node[anchor =  north ]  at (01) {\small{1}};
\node[anchor =  north ]  at (02) {\small{2}};
\node[anchor =  north ]  at (.4, -.04) {\ldots};
\node[anchor =  north ]  at (2.1, -.04) {\ldots};
\node[anchor =  north ]  at (023) {\small{23}};
\node[anchor =  north ]  at (029) {\small{29}};
\node[anchor =  south ]  at (151) {\small{52}};

\node[anchor = north]  at (0,-.15) {$e_2$};
\node[anchor = north]  at (2.3,-.15) {$\tau$};
\node[anchor = north]  at (2.9,-.15) {$e_1$};

\node[anchor = north]  at (5.15,1.3) {$\sigma$};



\end{tikzpicture}}
$$
    \caption{\label{fig:fullhasse} The complete Hasse diagram for the order complex $S_K$ from Figure \ref{fig:ordergradient}.}
\end{figure}

\begin{figure}
$$
\begin{tikzpicture}[scale = 1,decoration={markings,mark=at position 0.7 with {\arrow{Stealth}}},]

\node[inner sep=1.5pt, circle,  fill=red] (0) at (0,0) [draw] {};
\node[inner sep=1.5pt, circle, fill=black] (1) at (1,0) [draw] {};
\node[inner sep=1.5pt, circle, fill=black] (23) at (3,0) [draw] {};
\node[inner sep=1.5pt, circle, fill=black] (28) at (5,0) [draw] {};
\node[inner sep=1.5pt, circle, fill=red] (29) at (6,0) [draw] {};
\node[inner sep=1.5pt, circle, fill=black] (51) at (8,0) [draw] {};
\node[inner sep=1.5pt, circle, fill=black] (511) at (9,0) [draw] {};

\node[inner sep=1.5pt, circle, fill=black] (5) at (.5,2) [draw] {};
\node[inner sep=1.5pt, circle, fill=black] (225) at (2.2,2) [draw] {};
\node[inner sep=1.5pt, circle, fill=red] (231) at (3,2) [draw] {};
\node[inner sep=1.5pt, circle, fill=black] (235) at (3.8,2) [draw] {};
\node[inner sep=1.5pt, circle, fill=black] (285) at (5.5,2) [draw] {};
\node[inner sep=1.5pt, circle, fill=black] (295) at (6.5,2) [draw] {};
\node[inner sep=1.5pt, circle, fill=red] (52) at (8.5,2) [draw] {};
\node[inner sep=1.5pt, circle, fill=red] (515) at (10,2) [draw] {};

\draw[->-]  (1)--(5) node[midway, left] {};
\draw[->-]  (5)--(0) node[midway, left] {};
\draw[->-]  (23)--(225) node[midway, left] {};
\draw[->-]  (231)--(23) node[midway, left] {};
\draw[->-]  (235)--(23) node[midway, left] {};
\draw[->-]  (231)--(0) node[midway, left] {};
\draw[->-]  (285)--(28) node[midway, left] {};
\draw[->-]  (295)--(29) node[midway, left] {};
\draw[->-]  (511)--(285) node[midway, left] {};
\draw[->-]  (515)--(29) node[midway, left] {};
\draw[->-]  (515)--(511) node[midway, left] {};
\draw[->-]  (52)--(511) node[midway, left] {};
\draw[->-]  (52)--(51) node[midway, left] {};

\node[anchor =  north, red ]  at (0) {0};
\node[anchor =  north ]  at (1) {1};
\node[anchor =  north ]  at (23) {23};
\node[anchor =  north ]  at (28) {28};
\node[anchor =  north, red ]  at (29) {29};
\node[anchor =  north ]  at (51) {51};
\node[anchor =  north, fgreen ]  at (511) {51.1};

\node[anchor =  south, fgreen ]  at (5) {0.5};
\node[anchor =  south, fgreen ]  at (225) {22.5};
\node[anchor =  south, fgreen ]  at (231) {23.1};
\node[anchor =  south, fgreen ]  at (235) {23.5};
\node[anchor =  south, fgreen ]  at (285) {28.5};
\node[anchor =  south, fgreen ]  at (295) {29.5};
\node[anchor =  south, red ]  at (52) {52};
\node[anchor =  south, fgreen ]  at (515) {51.5};

\node[anchor = north]  at (0,-.3) {$e_2$};
\node[anchor = north]  at (6,-.3) {$e_1$};
\node[anchor = north]  at (3,-.3) {\large{$\tau$}};
\node[anchor = north]  at (8.5,2.7) {\large{$\sigma$}};

\node[anchor = north]  at (2,0) {\large{$\cdots$}};
\node[anchor = north]  at (4,0) {\large{$\cdots$}};
\node[anchor = north]  at (7,0) {\large{$\cdots$}};
\node[anchor = south]  at (1.35,2) {\large{$\cdots$}};

\end{tikzpicture}
$$
    \caption{\label{fig:relevant} The relevant portions of the Hasse diagram from Figure \ref{fig:fullhasse}.}
\end{figure}
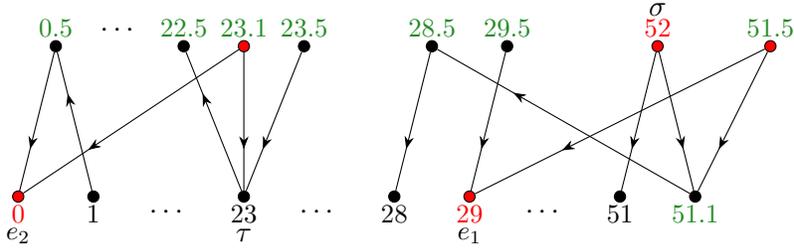

If $a$ is a real number, we have the {\em level subcomplex} of $S_K$:
$$S_K(a) = \bigcup_{\substack{F(\sigma)\le a\\ \sigma\in S_K}}\bigcup_{\tau<\sigma} \tau.$$ We also have the corresponding level subcomplexes of $\sd(X)$, which we denote by $\sd(X)(a)$. Observe that $S_K(a)\subseteq \sd(X)(a)$. Let $\sd(K)=\sd(X)\setminus \sd(T)$ and observe that each $S_K(a)$ lies in $\sd(K)$.   Now define a hypergraph $K(a)$ by $$K(a) = \sd(K)\cap \sd(X)(a).$$ Note that $K(a) = \sd(X)(a)\setminus \sd(T)(a)$.

The following result is the sublevel complex version of Proposition \ref{prop:strongdef}.

\begin{proposition}
    $|K(a)|$ deformation retracts to $|S_K(a)|$.
\end{proposition}

\begin{proof}
    Suppose $\sigma$ is a cell in $K(a)\setminus S_K(a)$. Then there are two cases.
    \begin{enumerate}
        \item $F(\sigma)\le a$. Since $\sigma\notin S_K(a)$, there is a face of $\sigma$ that does not lie in $\sd(K)$, and therefore $\sigma$ retracts to its boundary in $K(a)$.
        \item $F(\sigma)>a$, but there exists a coface $\tau>\sigma$ with $F(\tau)\le a$. This means that the pair $\{\sigma,\tau\}\in V_X'$. This implies that any other $\tau'>\sigma$ has $F(\tau')>a$ and any face $\eta<\sigma$ has $F(\eta)<F(\sigma)$. Thus $\sigma$ is a free face of $\tau$ in $\sd(X)(a)$ and we can therefore collapse the pair $\{\sigma,\tau\}$.
    \end{enumerate}
    In either case, we see that the cells not in $S_K(a)$ can be deformed away in $|K(a)|$  leaving $|S_K(a)|$ fixed. Proceeding by downward induction from the top dimensional cells in $K(a)$, the result follows.
\end{proof}

\begin{figure}
$$
\begin{tikzpicture}[scale = .46,decoration={markings,mark=at position 0.7 with {\arrow{Stealth}}},]

\filldraw[fill=black!30, draw=none] (0,0)--(15,0)--(15,15)--(105/22,225/22)--(7.5,7.5)--cycle;

\draw[thick] (0,0) -- (15,0);
\draw[thick] (0,0) -- (0,15);
\draw[ultra thick, red] (15,0) -- (15,15);
\draw[thick] (0,15) -- (15,15);

\draw[thick] (0,8) -- (15,15);
\draw[thick] (0,15) -- (15,0);
\draw[thick] (0,0) -- (7,15);
\draw[thick] (0,0) -- (15,15);
\draw[thick] (0,0) -- (15,8);
\draw[thick] (7,0) -- (15,15);

\draw[ultra thick, magenta] (105/22,225/22) -- (7.5,7.5) node[black,midway,sloped,above] {{\tiny 22.5}};
\draw[ultra thick, magenta] (7.5,7.5) -- (225/23,120/23) node[black,midway,sloped,above] {{\tiny 21.5}};
\draw[ultra thick, magenta] (225/23,120/23) -- (7,0) node[black,midway,sloped,below] {{\tiny 20.5}};
\draw[ultra thick, green] (105/22,225/22) -- (7,15) node[black,midway,sloped,above] {{\tiny 23.1}};
\draw[thick, black] (0,8) -- (105/22,225/22) node[midway,sloped,above] {{\tiny 23.25}};
\draw[thick] (105/22,225/22) -- (15,15) node[midway,sloped,above] {{\tiny 23.05}};
\draw[thick] (0,15) -- (105/22,225/22) node[midway,sloped,above] {{\tiny 23.3}};
\draw[thick] (0,0) -- (105/22,225/22) node[midway,sloped,below] {{\tiny 23.34}};
\draw[thick] (0,0) -- (7.5,7.5) node[midway,sloped,below] {{\tiny 22.6}};
\draw[thick] (7.5,7.5) -- (15,15) node[midway,sloped,below] {{\tiny 22.7}};

\filldraw[black] (105/22,225/22) circle (1pt) node[anchor=south]{{\tiny 23}};
\filldraw[black] (0,8) circle (1pt) node[anchor=east]{{\tiny 24}};
\filldraw[black] (7.5,7.5) circle (1pt) node[anchor=south]{{\tiny 22}};
\filldraw[black] (225/23,120/23) circle (1pt) node[anchor=south]{{\tiny 21}};
\filldraw[black] (7,0) circle (1pt) node[anchor=north]{{\tiny 20}};
\filldraw[black] (7,15) circle (1pt) node[anchor=south]{{\tiny 0}};

\filldraw[fill=none, draw=none] (0,8)--(0,15)--(105/22,225/22)--cycle node at (1.6,11.2) {{\tiny 23.4}};

\filldraw[fill=none, draw=none] (0,15)--(105/22,225/22)-- (7,15)--cycle node at (4,13.2) {{\tiny 23.2}};

\filldraw[fill=none, draw=none] (105/22,225/22)-- (7,15)-- (15,15) -- cycle node at (8,13.2) {{\tiny 23.09}};

\filldraw[fill=none, draw=none] (105/22,225/22)-- (15,15) -- (7.5,7.5) -- cycle node at (8,10.2) {{\tiny 22.9}};

\filldraw[fill=none, draw=none] (105/22,225/22)-- (0,0) -- (7.5,7.5) -- cycle node at (5,7.5) {{\tiny 23.33}};

\filldraw[fill=none, draw=none] (0,0) -- (7.5,7.5) -- (225/23,120/23) -- cycle node at (7.5,5.5) {{\tiny 22.4}};

\filldraw[fill=none, draw=none] (0,0) -- (0,8) -- (105/22,225/22) -- cycle node at (1.5,5.8) {{\tiny 23.35}};

\draw[thick, ->, blue] (259/44,3885/308) -- (281/44,3731/308);
\draw[thick, ->, blue] (0,11.5) -- (0.71,11.5);
\draw[thick, ->, blue] (0,4) -- (0.71,4);
\draw[thick, ->, blue] (0,8) -- (0.6,4347/525);
\draw[thick, ->, blue] (3.2,11.8) -- (3.7,12.3);
\draw[thick, ->, blue] (105/22,225/22) -- (116/22,214/22);
\draw[thick, ->, blue] (7.5,7.5) -- (8,7);
\draw[thick, ->, blue] (225/23,120/23) -- (427/46,1575/368);
\draw[thick, ->, blue] (3.2,48/7) -- (3.7,89/14);
\draw[thick, ->, blue] (4.5,4.5) -- (5,4);
\draw[thick, ->, blue] (9.5,9.5) -- (10,9);
\draw[thick, ->, blue] (7.5,805/70) -- (8,770/70);
\draw[thick, ->, blue] (5.5,88/30) -- (6,73/30);
\draw[thick, ->, blue] (13,104/15) -- (13,6.12);
\draw[thick, ->, blue] (12,3) -- (11.5,2.5);
\draw[thick, ->, blue] (12,75/8) -- (12.5,71/8);

\end{tikzpicture}
$$
\caption{\label{fig:sublevel} A portion of $K(23)$.}
\end{figure}
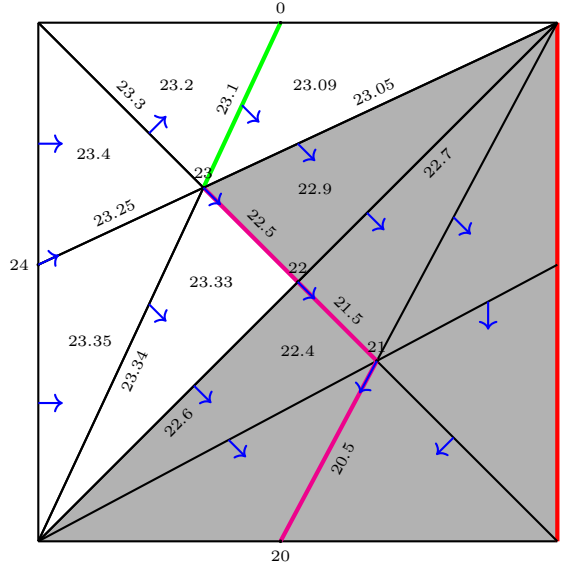

\begin{example}
    In Figure \ref{fig:sublevel} we show a portion of the open complex $K(23)$ from our running example. The level subcomplex $S_K(23)$ is shown in magenta and $K(23)$ is the shaded portion. Note that it contains the edge $23.05$. The edge $23.1$ (top portion of the diagram, in green) is paired in $V_X'$ but it is critical in $W$ (its endpoints have values $23$ and $0$). The retraction $|K(23)|\to |S_K(23)|$ is evident.
\end{example}

We now employ the results of discrete Morse theory to obtain the following, which is the analogue of the classical theorem in smooth Morse theory that sublevel sets of a Morse function are diffeomorphic if there are no critical values between them.

\begin{theorem}\label{thm:sublevel}
    If the interval $[a,b]$ does not contain any critical values of $F\colon S_K\to\zr$, then $|K(b)|$ deformation retracts to $|K(a)|$. If $[a,b]$ contains a single critical value $c$ of $F$, with corresponding critical cell $\sigma^{i}$, then $|K(b)|\simeq |K(a)|\cup \sigma^{i}$. 
\end{theorem}

\begin{proof}
    Suppose $[a,b]$ does not contain any critical values of $F$. Consider the commutative diagram
    $$\begin{tikzcd}
        \textrm{$|S_K(a)|$} \arrow[r] \arrow[d,"\simeq"] & \textrm{$|S_K(b)|$}\arrow[d,"\simeq"] \\
        \textrm{$|K(a)|$} \arrow[r] & \textrm{$|K(b)|$}
    \end{tikzcd}$$
    There is a deformation retraction $r:|S_K(b)|\to |S_K(a)|$ (in fact, a simplicial collapse $S_K(b)\searrow S_K(a)$) and therefore we conclude that $|K(b)|$ retracts to $|K(a)|$ leaving the latter fixed.

    For the second statement, note that if $\sigma$ is the critical $i$-cell with $F(\sigma)=c$, then $$S_K(b) = S_K(a)\cup_{\partial\sigma}\sigma.$$ We therefore conclude that $|K(b)|\simeq |K(a)|\cup \sigma$.
\end{proof}

\begin{corollary}\label{cor:htpytype}
    $|K|$ has the homotopy type of a cell complex with one cell of dimension $k$ for each critical cell of dimension $k$ in $W$. This contains a cell complex with one cell of dimension $i$ for each critical cell of height $i$ in $V_K$.
\end{corollary}

\begin{example}
    We conclude with a simple, but pathological, example to illustrate the usefulness of Corollary \ref{cor:htpytype}. Consider the $2$-simplex $X$ on the left side of Figure \ref{fig:pathological}. The subcomplex $T$ consists of a single vertex (in red, lower left) and the open complex $K$ is $X\setminus T$. A discrete gradient is shown on $K$. Note that there are {\em no critical cells}. However, we know that $|K|$ is not the empty set, and indeed it is contractible to a point. The order complex $S_K$ is shown on the right side of Figure \ref{fig:pathological}, along with the induced gradient $W$. Note that there are three critical vertices and two critical edges. This captures the topology of $|K|$. 
\end{example}

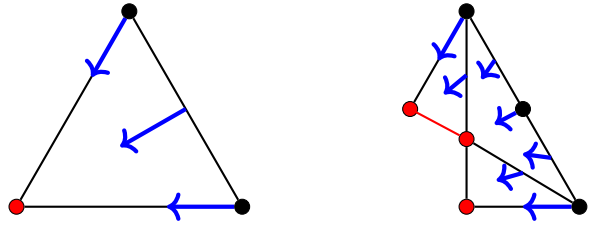
\begin{figure}
  \begin{center}  
\begin{tikzpicture}[scale = 1,decoration={markings,mark=at position 0.6 with {\arrow{triangle 60}}},]

\node[inner sep=2pt, circle, fill=red] (0) at (0,0) [draw] {};
\node[inner sep=2pt, circle, fill=black] (1) at (3,0) [draw] {};
\node[inner sep=2pt, circle, fill=black] (2) at (1.5,2.6) [draw] {};

\draw[thick, draw = black]  (0)--(1) node[midway, left] {};
\draw[thick, draw = black]  (0)--(2) node[midway, left] {};
\draw[thick,draw = black] (2)--(1) node[midway, right] {};

\draw[ultra thick, ->, blue] (1)--(2,0);
\draw[ultra thick, ->, blue] (2)--(1,1.73);
\draw[ultra thick, ->, blue] (4.5/2,2.6/2)--(1.38,0.8);

\end{tikzpicture} 
\hskip .75in
\begin{tikzpicture}[scale = 1,decoration={markings,mark=at position 0.6 with {\arrow{triangle 60}}},]

\node[inner sep=2pt, circle, fill=black] (1) at (3,0) [draw] {};
\node[inner sep=2pt, circle, fill=black] (2) at (1.5,2.6) [draw] {};
\node[inner sep=2pt, circle, fill=red] (3) at (1.5,.9) [draw] {};
\node[inner sep=2pt, circle, fill=black] (4) at (4.5/2,2.6/2) [draw] {};
\node[inner sep=2pt, circle, fill=red] (5) at (0.75,1.3) [draw] {};
\node[inner sep=2pt, circle, fill=red] (6) at (1.5,0) [draw] {};

\draw[thick, draw = black]  (6)--(1) node[midway, left] {};
\draw[thick, draw = black]  (5)--(2) node[midway, left] {};
\draw[thick, draw = black]  (4)--(3) node[midway, left] {};
\draw[thick, draw = black]  (1)--(3) node[midway, right] {};
\draw[thick, draw = black]  (2)--(3) node[midway, right] {};
\draw[thick,draw = black] (2)--(1) node[midway, right] {} ;
\draw[thick,draw = red] (6)--(3) node[midway, right] {};
\draw[thick,draw = red] (1.5/2,2.6/2)--(3) node[midway, right] {};

\draw[thick, ->, blue] (4)--(1.875,1.1) node[]{};
\draw[thick, ->, blue] (1)--(2.25,0) node[]{};
\draw[thick, ->, blue] (2)--(1.125,1.95) node[]{};
\draw[thick, ->, blue] (1.875,1.95)--(1.7,1.7) node[]{};
\draw[thick, ->, blue] (2.625,0.65)--(2.25,0.70) node[]{};
\draw[thick, ->, blue] (2.25,0.45)--(1.9,0.35) node[]{};
\draw[thick, ->, blue] (1.5,1.75)--(1.2,1.5) node[]{};

\end{tikzpicture}
\end{center}
    \caption{\label{fig:pathological} A gradient vector field with no critical cells (left), and the associated order complex (right).}
\end{figure}

\subsection*{Acknowledgments} We are grateful to an anonymous referee for many helpful suggestions.

\bibliographystyle{amsplain}
\bibliography{openDMT}

\appendix
\section{\label{apdx} Pseudocode for algorithms}

Algorithm \ref{alg: main} describes the process for generating a discrete gradient on the order complex $S_K$. 

\begin{definition}\label{def:maxintlabel}
Let $X$ be a simplicial complex and $\sd(X)$ the barycentric subdivision of $X$. Let $\gamma \in \sd(X)$ correspond to the chain $\sigma_0\subseteq \sigma_1\subseteq \cdots \subseteq \sigma_k$. The \textbf{maximal element} of $\gamma$ is $\max(\gamma):= \sigma_k$.  For any $\sigma \in X$, the \textbf{interior simplices} of $\sigma$ in $\sd(X)$ are
$$
\sd(\sigma)^{\circ}:=\{\gamma \in \sd(X): \max(\gamma)=\sigma\}.
$$
The \textbf{label} of $\gamma$ is the ordered partition 
$$
\lambda(\gamma)=(I_1 I_2\ldots I_{k+1})
$$
where $I_1=\sigma_0$ and $I_j=\sigma_{j-1}\setminus \sigma_{j-2}$ for $j\geq 2$.
\end{definition}

\begin{definition}\label{def:lcs}
The \textbf{longest common suffix} of two labels is the number of consecutive positions, starting from the rightmost position, such that the sets in the label at each position are equal.
\end{definition}

The following is the pseudocode for the algorithm in Section \ref{sec: structure theorem}. The routine \texttt{Critical} is described in Algorithm \ref{alg: critical} and the routine \texttt{RegularPair} is described in \ref{alg: regularpair}.

\begin{algorithm}
  \caption{Main Open discrete Morse function algorithm}\label{alg: main}
  \begin{algorithmic}
\State \verb"Input:" Open simplicial complex $K=X\setminus T$ with gradient $V_K$ 
\State \verb"Output:" Discrete Morse function $F\colon S_K\to \mathbb{R}$ and gradient vector field $W=V_F$
\end{algorithmic}
\begin{algorithmic}[1]
\State Extend $V_K$ to $V_X$ by starting with $V_K$ and making everything in $T$ critical
\State Choose an injective $f\colon X \to \mathbb{R}$ satisfying $V_f=V_X$
\ForAll {$\sigma \in K$} 
\State $\mathrm{ht}(\sigma)\leftarrow $ length of longest chain in $K$ ending at $\sigma$
\EndFor
\State Initialize $V'_X\leftarrow \emptyset$
\ForAll { $\sigma \in X$, in order of increasing dimension} 
\If {$\sigma$ is critical in $V_X$} 
\State Apply \verb"Critical"$(\sigma, f)$ and add resulting pairs to $V'_X$
\ElsIf { $\sigma=\beta$ is part of regular pair $(\alpha, \beta)\in V_X$} 
\State Apply \verb"RegularPair"$(\alpha, \beta, V_K, f)$ and add resulting pairs to $V'_X$ 
\ElsIf  { $\sigma=\alpha$ is part of regular pair $(\alpha, \beta)\in V_X$} 
\State skip
\EndIf
\EndFor
\State $W\leftarrow \{(\gamma, \delta)\in V'_X: \gamma, \delta \in S_K\}$
\State $\gamma \in S_K$ is critical in $W$ if it
 is critical in $V'_X$ or $\gamma$'s partner in $V'_X$ is not in $S_K$
\State Extend $f$ to $F$ on $\mathrm{sd}(X)$ compatible with $V'_X$, then restrict to get $F:S_K\to\zr$.
  \end{algorithmic}
\end{algorithm}

We now present pseudocode for Zhukova's algorithm \cite{Zhukova2017}, along with our modifications as described in Section \ref{sec: structure theorem}.

\begin{definition}\label{def:amenable}
Let $\sigma$ be an $n$-simplex of height $i$ in $K=X\setminus T$. An $n$-cell $\alpha\in\textrm{sd}(\sigma)^\circ$ is called \textbf{amenable} if it has an $i$-dimensional face $\tau$ lying in $S_K$. We may then write $\tau=b\ast\tau'$, where $b$ is the barycenter of $\sigma$ and $\tau'$ is an $(i-1)$-dimensional simplex in the subdivision of a simplex $\sigma'<\sigma$ of height $i-1$. 
\end{definition}

In Figure \ref{fig:guarantee} the $2$-simplex $\tau'$ is amenable, while the cell $\tau$ is not. The $2$-simplex sharing edge $f$ with $\tau'$ is also amenable.

\begin{definition}\label{def:admissible}
Suppose we have a discrete gradient $V_X$ with a compatible injective discrete Morse function $f:X\to\zr$. Suppose the $n$-simplex $\sigma$ in $K=X\setminus T$ has height $i$. We call an ordering of the vertices of $\sigma$ \textbf{admissible} if the  $n$-cell $\alpha\in\text{sd}(\sigma)^\circ$ with label  $\lambda(\alpha)=(\{n+1\}\{n\}\ldots \{2\}\{1\})$ is amenable and has an $i$-dimensional face $\tau$ lying in $S_K$ minimizing $f(\sigma)-f(\sigma')$ among all such $\tau$.
\end{definition}

 \begin{algorithm}[H]
  \caption{Critical}\label{alg: critical}
  \begin{algorithmic}
 \State \verb"Input:" A critical simplex $\sigma^{(n)}$ in $V_X$ and compatible discrete Morse function $f\colon X\to \mathbb{R}$ 
 \State \verb"Output:" A matching on $\sd(\sigma)^{\circ}$ with exactly one critical $n$-simplex 
\end{algorithmic}
\begin{algorithmic}[1]
\ForAll {$\sigma \in K$}
\State $i=\mathrm{ht}(\sigma)$

\If {$i=n$} 
\State choose any vertex ordering of $\sigma$
\Else { choose an admissible vertex ordering of $\sigma$ }
\EndIf
\State Set $\alpha\in \mathrm{sd}(\sigma)^{\circ}$ to be the $n$-simplex having label $$\lambda(\alpha)=(\{n+1\}\{n\}\ldots \{2\}\{1\})$$
\State Leave $\alpha$ unpaired
\ForAll {$\gamma\in \mathrm{sd}(\sigma)^{\circ}$ with $\gamma\neq \alpha$ unpaired and $\lambda(\gamma)=(I_1\ldots I_{k+1})$ }
 \State let $\ell=$ length of longest common suffix of $\lambda(\gamma)$ and $\lambda(\alpha)$
 \State let vertex $v$ have label $\ell+1$
 \State find $I_j$ containing vertex $v$
\If {$I_j=\{v\}$ is a singleton} 
 \State pair with simplex having label $(I_1\ldots I_{j-1}(I_j\cup I_{j+1})I_{j+2}\ldots I_{k+1})$
\Else { pair with simplex having label $(I_1\ldots I_{j-1}\{v\}(I_j\setminus \{v\})I_{j+1}\ldots I_{k+1})$}
\EndIf
\EndFor
\EndFor
  \end{algorithmic}
\end{algorithm}

\begin{remark}
    Note that all simplices $\gamma$ in the boundary of the critical cell $\alpha$ are paired with something of larger dimension outside of $\alpha$ by this algorithm. Thus, the $i$-simplex $\tau$ lying in $S_K$ is paired with an $(i+1)$-simplex, which cannot lie in $S_K$ by definition.
\end{remark}

 \begin{algorithm}[H]
  \caption{RegularPair}\label{alg: regularpair}
  \begin{algorithmic}
\State \verb"Input:" A regular pair $(\alpha^{n-1}, \beta^n)$, gradient vector field $V_K$, and compatible discrete Morse function $f\colon X\to \mathbb{R}$ 

\State\verb"Output:" A matching on $\sd(\alpha)^{\circ}\cup \sd(\beta)^{\circ}$
\end{algorithmic}
\begin{algorithmic}[1]
\State Let $S=\{\alpha' : \alpha' \subseteq \beta, \alpha'\neq \alpha, \alpha' \text{ is a tail of a pair in } V_K\}$
\If {$S=\emptyset$} 
\State choose any ordering satisfying $\alpha=\{1,2, \ldots, n\}$
    \ElsIf {$S\neq \emptyset$}
    \State find $\alpha'\in S$ maximizing $f(\beta)-f(\alpha')$ and choose ordering $\alpha'=\{1,2, \ldots, n\}$
    \ForAll {$\gamma\in \sd(\alpha)^{\circ}$ with label $\lambda(\gamma)=(I_1\ldots I_{k+1})$} 
    \State pair $\gamma$ with simplex that has label $(I_1 \ldots I_{k+1} \{n+1\})$
    \EndFor
    \ForAll {unpaired $\gamma\in \sd(\beta)^{\circ}$ with label $\lambda(\gamma)=(I_1\ldots I_{k+1})$} 
    \State find $I_j$ containing vertex $n+1$
    \If {$I_j=\{n+1\}$ is a singleton at last position} 
    \State pair with $(I_1 \ldots I_k)$
    \ElsIf {$I_j=\{n+1\}$ is a singleton not at last position} 
    \State pair with $(I_1 \ldots I_{j-1} (I_j\cup I_{j+1})I_{j+2} \ldots I_{k+1})$
    \ElsIf {$I_j$ is not a singleton containing $n+1$} 
    \State pair with $(I_1 \ldots I_{j-1} \{n+1\} (I_j\setminus \{n+1\})I_{j+1} \ldots I_{k+1})$
    \EndIf
    \EndFor
\EndIf
\If {$S\neq \emptyset$} 
    \State let $b=v_{\beta}$ and $z=v_{\alpha'}$
    \State remove pairs: $(b, (b, v_{n+1}))$ and $((b,z), (b,z, v_{n+1}))$
    \State add pairs: $(b, (b,z))$ and $(v_{n+1}, (b,z, v_{n+1}))$
 \EndIf
  \end{algorithmic}
\end{algorithm}

\end{document}